\documentclass[amstex,12pt]{article}
\usepackage{amssymb,amsmath,amsthm,graphicx}
\newtheorem{theorem}{Theorem}[section]

\newtheorem{lemma}{Lemma}[section]
\newtheorem{definition}{Definition}[section]
\newtheorem{remark}{Remark}[section]
\newtheorem{example}{Example}[section]

\newcommand{\beq}{\begin{equation}}
\newcommand{\eeq}{\end{equation}}
\newcommand{\beqn}{\begin{eqnarray}}
\newcommand{\eeqn}{\end{eqnarray}}

\allowdisplaybreaks \baselineskip 20pt
\begin{document}

\title{Existence and exponential stability of positive almost
periodic solution for Nicholson's blowflies models on time
scales\thanks{This work is supported by the National Natural Sciences Foundation of People's Republic of China under Grant
11361072.}}
\author{Yongkun Li\thanks{%
The corresponding author.   Email:
yklie@ynu.edu.cn.} and Bing Li\\
$^a$Department of Mathematics, Yunnan University\\
Kunming, Yunnan 650091,
 People's Republic of China}
\date{}
\maketitle \allowdisplaybreaks
\begin{abstract}
In this paper, we first give a new  definition  of almost periodic time scales, two new definitions of almost periodic functions on time scales and  investigate
some  basic properties of them.
Then, as an application, by using the fixed point theorem in Banach space and the
time scale calculus theory, we obtain some sufficient conditions for
the existence and  exponential stability of  positive almost periodic solutions
for a class of Nicholson's blowflies models on time scales.
Finally, we present an illustrative example to show the
effectiveness of obtained results. Our results show that under a simple condition the
continuous-time Nicholson's blowflies models and their discrete-time
analogue have the same dynamical behaviors.
\end{abstract}

\allowdisplaybreaks
\section{Introduction}

 \setcounter{section}{1}
\setcounter{equation}{0}
 \indent

To describe the population of the Australian sheep-blowfly and to
agree with the experimental data obtained in \cite{1}, Gurney et
al. \cite{2} proposed the following delay differential equation model:
\begin{equation}\label{e0}
x'(t)=-\delta x(t)+px(t-\tau)e^{-ax(t-\tau)},
\end{equation}
where $p$ is the maximum
per capita daily egg production rate, $1/a$
 is the size at which the blowfly population reproduces at its maximum rate, $\delta$ is the per capita
daily adult death rate, and $\tau$ is the generation time.
Since equation \eqref{e0} explains Nicholson's data of blowfly more
accurately, the model and its modifications have been now refereed
to as Nicholson's Blowflies model.   The theory
of the Nicholsons blowflies equation has made a remarkable progress in
the past forty years with main results scattered in numerous research papers.
Many important results on the qualitative properties of the model such as existence of
positive solutions, positive periodic  or positive almost periodic solutions, persistence, permanence, oscillation
and stability for the classical Nicholson¡¯s model and its generalizations   have been established in the
literature \cite{v5,v12,liu2,v24,v32,3,4,liu,5,6}.
 For
example, to describe the models of marine protected areas and B-cell
chronic lymphocytic leukemia dynamics that are examples of
Nicholson-type delay differential systems, Berezansky et al.
\cite{7} and Wang et al. \cite{8} studied the following
Nicholson-type delay system:
\begin{eqnarray*}
\left\{
\begin{array}{lll}
N'_1(t)
=-\alpha_1(t)N_1(t)+\beta_1(t)N_2(t)+\sum\limits^{m}_{j=1}c_{1j}(t)N_1(t-\tau_{1j}(t))e^{-\gamma_{ij}(t)N_1(t-\tau_{1j}(t))}
,\\
N'_2(t)
=-\alpha_2(t)N_2(t)+\beta_2(t)N_1(t)+\sum\limits^{m}_{j=1}c_{2j}(t)N_2(t-\tau_{1j}(t))e^{-\gamma_{ij}(t)N_2(t-\tau_{1j}(t))}
,
\end{array}
\right.
\end{eqnarray*}
where $\alpha_i, \beta_i, c_{ij}, \gamma_{ij}, \tau_{ij}\in
C(\mathbb{R}, (0,+\infty))$, $i=1,2, j=1,2,\ldots,m$; in \cite{9},
the authors discussed some aspects of the global dynamics for a
Nicholson's blowflies model with patch structure given by
\[
x'_i(t)=-d_ix_i(t)+\sum\limits_{j=1}^na_{ij}x_j(t)
+\sum\limits_{j=1}^m\beta_{ij}x_i(t-\tau_{ij})e^{-x_i(t-\tau_{ij})},
i=1,2,\ldots,n.
\]

In the real world phenomena, since the
almost periodic variation of the environment plays a crucial role in many biological and
ecological dynamical systems and is more frequent and general than the periodic variation
of the environment. Hence, the effects of almost periodic environment on evolutionary
theory have been the object of intensive analysis by numerous authors and some
of these results for Nicholson¡¯s blowflies models  can be found in \cite{10,11,12,13,v18, ni1,ni2,ni3}.

Besides, although most models are described by differential
equations, the discrete-time models governed by difference equations
are more appropriate than the continuous ones when the size of the
population is rarely small, or the population has non-overlapping
generations. Hence, it is also important to study the dynamics of
discrete-time Nicholson's blowflies models. Recently, authors of
\cite{14,15} studied the existence and exponential convergence of
almost periodic solutions for discrete Nicholson's blowflies models,
respectively. In fact, it is troublesome to study the dynamics for
discrete and continuous systems respectively, therefore, it is
significant to study that on time scales, which was initiated by
Stefan Hilger (see \cite{16}) in order to unify continuous and
discrete cases. However, to the best of our knowledge, very few results are
available on the existence and stability of  positive almost periodic
solutions for Nicholson's blowflies models on time scales except
\cite{17}. But \cite{17} only considered the asymptotical stability
of the model and the exponential stability is stronger than
asymptotical stability among different stabilities.

On the other hand, in order to study the almost periodic dynamic equations on time scales, a concept of
almost periodic time scales was proposed in \cite{li1}. Based on this concept, almost periodic
functions \cite{li1}, pseudo almost periodic functions \cite{li3},
almost automorphic functions \cite{r1},   weighted pseudo
almost automorphic functions \cite{li2}, weighted piecewise pseudo almost
automorphic functions \cite{w1} and  almost periodic set-valued functions \cite{set} on  on time scales were defined successively.  Also, some works have been done under the
concept of almost periodic time scales (see \cite{r2,r5,r6,r11,r12,r13,r14,r15}). Although  the concept of
almost periodic time scales in \cite{li1} can unify the
continuous and discrete situations effectively, it is very restrictive. This
excludes many interesting time scales.
Therefore, it is a challenging and important problem in
theories and applications to find  new concepts of periodic time scales (\cite{liwang,wa,lili1,lili3,lili2}).

Motivated by the above discussion, our main purpose of  this paper is firstly to propose
a new   definition  of almost periodic time scales, two new definitions of almost periodic functions on time scales  and study
some basic properties of them.
Then, as an application, we study  the existence and global exponential
stability of positive almost periodic solutions for the following Nicholson's
blowflies model  with patch structure and multiple time-varying
delays on time scales:
\begin{eqnarray}\label{e1}
x_i^\Delta(t)&=&-c_i(t)x_i(t)+\sum\limits_{k=1,k\neq
i}^nb_{ik}(t)x_k(t)\nonumber\\
&&+\sum\limits_{j=1}^n\beta_{ij}(t)x_i(t-\tau_{ij}(t))e^{-\alpha_{ij}(t)x_i(t-\tau_{ij}(t))},\,\,
i=1,2,\ldots,n,
\end{eqnarray}
where $t\in \mathbb{T}$, $\mathbb{T}$ is an almost periodic time
scale, $x_i(t)$ denotes the density of the species in patch $i$,
$b_{ik}(k\neq i)$ is the migration coefficient from patch $k$ to
patch $i$ and the natural growth in each patch is of Nicholson-type.

For convenience, for a positive almost periodic function
$f:\mathbb{T}\rightarrow\mathbb{R}$, we denote
$f^+=\sup\limits_{t\in\mathbb{T}}f(t),
f^-=\inf\limits_{t\in\mathbb{T}}f(t)$. Due to the biological meaning
of \eqref{e1}, we just consider the following initial condition:
\begin{eqnarray}\label{e2}
\varphi_i(s)>0,\,\,s\in
[t_0-\theta,t_0]_{\mathbb{T}},\,t_0\in\mathbb{T},\,i=1,2,\ldots,n,
\end{eqnarray}
where
$\theta=\max\limits_{(i,j)}\sup\limits_{t\in\mathbb{T}}\{\tau_{ij}(t)\}$,
$[t_0-\theta,t_0]_{\mathbb{T}}=[t_0-\theta,t_0]\cap \mathbb{T}$.

This paper is organized as follows: In Section 2, we introduce some
notations and definitions  which
are needed in later sections. In Section 3, we  give a new  definition of almost periodic time scales and two new definitions of almost periodic functions on time scales, and we state and  prove
some basic properties of them. In Section 4, we establish some
sufficient conditions for the existence and  exponential stability of
positive almost periodic solutions of \eqref{e1}. In Section 5, we give
an example to illustrate the feasibility of our results obtained in
previous sections. We draw a conclusion in Section 6.

\section{ Preliminaries}

  \setcounter{equation}{0}

\indent

In this section, we shall first recall some definitions and state some results which are used in what follows.

Let $\mathbb{T}$ be a nonempty closed subset (time scale) of
$\mathbb{R}$. The forward and backward jump operators $\sigma,
\rho:\mathbb{T}\rightarrow\mathbb{T}$ and the graininess
$\mu:\mathbb{T}\rightarrow\mathbb{R}^+$ are defined, respectively,
by
\[
\sigma(t)=\inf\{s\in\mathbb{T}:s>t\},\,\,\,\,
\rho(t)=\sup\{s\in\mathbb{T}:s<t\}\,\,\,\,{\rm
and}\,\,\,\,\mu(t)=\sigma(t)-t.
\]

A point $t\in\mathbb{T}$ is called left-dense if $t>\inf\mathbb{T}$
and $\rho(t)=t$, left-scattered if $\rho(t)<t$, right-dense if
$t<\sup\mathbb{T}$ and $\sigma(t)=t$, and right-scattered if
$\sigma(t)>t$. If $\mathbb{T}$ has a left-scattered maximum $m$,
then $\mathbb{T}^k=\mathbb{T}\setminus\{m\}$; otherwise
$\mathbb{T}^k=\mathbb{T}$. If $\mathbb{T}$ has a right-scattered
minimum $m$, then $\mathbb{T}_k=\mathbb{T}\setminus\{m\}$; otherwise
$\mathbb{T}_k=\mathbb{T}$.

A function $f:\mathbb{T}\rightarrow\mathbb{R}$ is right-dense
continuous provided it is continuous at right-dense point in
$\mathbb{T}$ and its left-side limits exist at left-dense points in
$\mathbb{T}$. If $f$ is continuous at each right-dense point and
each left-dense point, then $f$ is said to be continuous function on
$\mathbb{T}$.

For $y:\mathbb{T}\rightarrow\mathbb{R}$ and $t\in\mathbb{T}^k$, we
define the delta derivative of $y(t)$, $y^\Delta(t)$, to be the
number (if it exists) with the property that for a given
$\varepsilon>0$, there exists a neighborhood $U$ of $t$ such that
\[
|[y(\sigma(t))-y(s)]-y^\Delta(t)[\sigma(t)-s]|<\varepsilon|\sigma(t)-s|
\]
for all $s\in U$.

If $y$ is continuous, then $y$ is right-dense continuous, and if $y$
is delta differentiable at $t$, then $y$ is continuous at $t$.

Let $y$ be right-dense continuous. If $Y^{\Delta}(t)=y(t)$, then we
define the delta integral by
$
\int_a^{t}y(s)\Delta s=Y(t)-Y(a).
$

A function $r:\mathbb{T}\rightarrow\mathbb{R}$ is called regressive
if
$
1+\mu(t)r(t)\neq 0
$
for all $t\in \mathbb{T}^k$. The set of all regressive and
$rd$-continuous functions $r:\mathbb{T}\rightarrow\mathbb{R}$ will
be denoted by $\mathcal{R}=\mathcal{R}(\mathbb{T})=\mathcal{R}(\mathbb{T},\mathbb{R})$. We
define the set
$\mathcal{R}^+=\mathcal{R}^+(\mathbb{T},\mathbb{R})=\{r\in\mathcal{R}:1+\mu(t)r(t)>0,\,\,\forall
t\in\mathbb{T}\}$.

\begin{lemma}$(\cite{18})$  Suppose that $p\in \mathcal{R}^{+}$, then
\begin{itemize}
     \item  [$(i)$]$e_{p}(t,s)>0$, for all $t,s\in\mathbb{T}$;
    \item[$(ii)$]if $p(t)\leq q(t)$ for all $t\geq s, t,s\in\mathbb{T}$, then
$e_{p}(t,s)\leq e_{q}(t,s)$ for all $t \geq s$.
\end{itemize}
\end{lemma}

\begin{definition}\cite{fink} A subset $S$ of $\mathbb{R}$ is called relatively dense if there exists a positive
number $L$ such that $[a, a + L] \cap S  \neq \phi$ for all $a\in \mathbb{R}$. The number $L$ is called the inclusion
length.
\end{definition}
\begin{definition}\label{def21}\cite{li1}  A time scale $\mathbb{T}$ is called an almost periodic time scale if
\begin{eqnarray*}
\Pi=\big\{\tau\in\mathbb{R}: t\pm\tau\in\mathbb{T}, \forall t\in{\mathbb{T}}\big\}\neq \{0\}.
\end{eqnarray*}
\end{definition}

The following definition is a slightly modified version of Definition 3.10 in \cite{li1}.
\begin{definition}\label{defli2}
Let $\mathbb{T}$ be an almost periodic time scale. A function  $f\in
C(\mathbb{T}\times D,\mathbb{E}^n)$ is  called an almost
periodic function in $t\in \mathbb{T}$ uniformly for $x\in D$ if the
$\varepsilon$-translation set of $f$
$$E\{\varepsilon,f,S\}=\{\tau\in\Pi:|f(t+\tau,x)-f(t,x)|<\varepsilon,\,\,
\forall (t,x)\in   \mathbb{T}\times S\}$$ is relatively
dense  for all $\varepsilon>0$ and   for each
compact subset $S$ of $D$; that is, for any given $\varepsilon>0$
and each compact subset $S$ of $D$, there exists a constant
$l(\varepsilon,S)>0$ such that each interval of length
$l(\varepsilon,S)$ contains   a $\tau(\varepsilon,S)\in
E\{\varepsilon,f,S\}$ such that
\begin{equation*}
|f(t+\tau,x)-f(t,x)|<\varepsilon, \,\,\forall (t,x)\in
\mathbb{T}\times S.
\end{equation*}
   $\tau$ is called the $\varepsilon$-translation number of $f$.
\end{definition}

\section{Almost periodic time scales and almost periodic functions on time scales}
  \setcounter{equation}{0}
\indent

In this section, we will give a new definition  of almost periodic time scales and two new definitions of almost periodic functions on time scales, and we  will investigate
some basic properties of them. Our   new  definition of almost periodic time scales is as follows:
\begin{definition}\label{def31}
A time scale $\mathbb{T}$ is called an almost periodic time scale if the set
$$\Pi:=\big\{\tau\in \mathbb{T}:\mathbb{T}_\tau\neq \emptyset\quad and \quad\mathbb{T}_\tau\neq \{0\}\big\}\neq\{0\},$$
where $\mathbb{T}_\tau=\mathbb{T}\cap\{\mathbb{T}-\tau\}=\mathbb{T}\cap \{t-\tau: t\in \mathbb{T}\}$,  satisfies
\begin{itemize}
  \item [$(i)$]   $\Pi\neq \emptyset$,
  \item [$(ii)$] if $\tau_1,\tau_2\in \Pi,$ then $\tau_1\pm \tau_2\in \Pi$,
  \item [$(iii)$] $\widetilde{\mathbb{T}}:=\mathbb{T}(\Pi)=\bigcap\limits_{\tau\in\Pi}\mathbb{T}_\tau\neq \emptyset$.
\end{itemize}
\end{definition}
Clearly, if $t\in \mathbb{T}_\tau$, then $t+\tau\in \mathbb{T}.$ If $t\in \widetilde{\mathbb{T}}$, then $t+\tau\in \mathbb{T}$ for $\tau\in \Pi$.
\begin{remark}\label{ar1}
Obviously,  if $\mathbb{T}$ is an almost periodic time scale under Definition \ref{def31}, then
$\inf\mathbb{T}=-\infty$ and $\sup\mathbb{T}=+\infty.$ If $\mathbb{T}$ is an almost periodic time scale under Definition \ref{def21}, then $\mathbb{T}$ is also an almost periodic time scale under Definition \ref{def31}
and in
this case, $\widetilde{\mathbb{T}}=\mathbb{T}$.
\end{remark}

\begin{example}
Let $\mathbb{T}=\mathbb{Z}\cup\{\frac{1}{4}\}$. For every $\tau\in\mathbb{Z}$, we have $\mathbb{T}_{\tau}=\mathbb{Z}$
and $\mathbb{T}_{\frac{1}{4}}=\{0\}$. Hence $\Pi=\mathbb{Z}$ and $\widetilde{\mathbb{T}}=\bigcap_{\tau\in\Pi}\mathbb{T}_{\tau}=\mathbb{Z}\neq\emptyset$. So, $\mathbb{T}$
is an almost periodic time scale under Definition \ref{def31} but it is not an almost periodic time scale under Definition \ref{def21}.
\end{example}

\begin{lemma}\label{lem31}
If $\mathbb{T}$ is an almost periodic time scales under Definition \ref{def31}, then  $\widetilde{\mathbb{T}}$ is an almost periodic time scale under Definition \ref{def21}.
\end{lemma}
\begin{proof}
By contradiction, suppose that there exists a $t_0\in \widetilde{\mathbb{T}}$ such that for every $\tau\in\Pi\setminus\{0\}$, $t_0+\tau\notin \widetilde{\mathbb{T}}$ or $t_0-\tau\notin \widetilde{\mathbb{T}}$.

\textbf{Case (i)}
If $t_0+\tau\notin \widetilde{\mathbb{T}}$, then there exists a ${\tau_{t_0}}\in \Pi$
such that $t_0+\tau\notin\mathbb{T}_{{\tau_{t_0}}}$. On  one hand, since $t_0+\tau\in\mathbb{T}$,
$t_0+\tau+\tau_{t_0}\notin\mathbb{T}$. On the other hand, since $t_0\in \widetilde{\mathbb{T}}$ and $\tau+\tau_{t_0}\in\Pi$,
$t_0+\tau+\tau_{t_0}\in\mathbb{T}$. This is a contradiction.

\textbf{Case (ii)}
If $t_0-\tau\notin \widetilde{\mathbb{T}}$, then there exists a ${\tilde{\tau}_{t_0}}\in \Pi$
such that $t_0-\tau\notin\mathbb{T}_{{\tilde{\tau}_{t_0}}}$. On  one hand, since $t_0-\tau\in\mathbb{T}$,
$t_0-\tau+\tilde{\tau}_{t_0}\notin\mathbb{T}$. On the other hand, since $t_0\in \widetilde{\mathbb{T}}$ and $-\tau+\tilde{\tau}_{t_0}\in\Pi$,
$t_0-\tau+\tilde{\tau}_{t_0}\in\mathbb{T}$. This is a contradiction.

Therefore,   for every $t\in \widetilde{\mathbb{T}}$, there exists a $\tau\in \Pi\setminus\{0\}$ such that $t\pm\tau\in \widetilde{\mathbb{T}}$.
Hence, $T$ is an almost periodic time scale under Definition \ref{def21}. The proof is complete.
\end{proof}

Throughout this section, $\mathbb{E}^{n}$ denotes $\mathbb{R}^{n}$ or
$\mathbb{C}^{n}$, $D$ denotes an open set in $\mathbb{E}^{n}$ or
$D=\mathbb{E}^{n}$, $S$ denotes an arbitrary compact subset of $D$.

 From \cite{li1}, under Definitions \ref{def21} and \ref{defli2}, we know that  if we
 denote by $BUC(\mathbb{T}\times D,\mathbb{R}^{n})$ the
collection of all bounded uniformly continuous functions from $\mathbb{T}\times D$ to $\mathbb{R}^{n}$, then
\begin{equation}\label{g1}
AP(\mathbb{T}\times D,\mathbb{R}^{n})\subset BUC(\mathbb{T}\times D,\mathbb{R}^{n}),
\end{equation}
where $AP(\mathbb{T}\times D,\mathbb{R}^{n})$  are  the collection of all almost periodic functions in $t\in \mathbb{T}$ uniformly for $x\in D$. It is well known that  if we let $\mathbb{T}=\mathbb{R}$ or $\mathbb{Z}$, \eqref{g1} is  valid.
So, for simplicity, we  give the following definition:
\begin{definition}\label{def32}
Let $\mathbb{T}$ be an almost periodic time scale under sense of Definition \ref{def31}. A function  $f\in
BUC(\mathbb{T}\times D,\mathbb{E}^n)$  is  called an almost
periodic function in $t\in \mathbb{T}$ uniformly for $x\in D$ if the
$\varepsilon$-translation set of $f$
$$E\{\varepsilon,f,S\}=\{\tau\in\Pi:|f(t+\tau,x)-f(t,x)|<\varepsilon,\quad\forall  (t,x)\in   \widetilde{\mathbb{T}}\times S\}$$ is  relatively
dense  for all $\varepsilon>0$ and   for each
compact subset $S$ of $D$; that is, for any given $\varepsilon>0$
and each compact subset $S$ of $D$, there exists a constant
$l(\varepsilon,S)>0$ such that each interval of length
$l(\varepsilon,S)$ contains   a $\tau(\varepsilon,S)\in
E\{\varepsilon,f,S\}$ such that
\begin{equation*}
|f(t+\tau,x)-f(t,x)|<\varepsilon, \quad\forall (t,x)\in
\widetilde{\mathbb{T}}\times S.
\end{equation*}
This   $\tau$ is called the $\varepsilon$-translation number of $f$.
\end{definition}
\begin{remark}
If  $\mathbb{T}=\mathbb{R}$, then $\widetilde{\mathbb{T}}=\mathbb{R}$, in
this case,  Definition \ref{def32} is actually equivalent to the definition of the uniformly almost periodic functions in Ref.
\cite{fink}.
If $\mathbb{T}=\mathbb{Z}$, then $\widetilde{\mathbb{T}}=\mathbb{Z}$, in
this case,  Definition \ref{def32} is actually equivalent to the definition of the uniformly almost periodic sequences  in Refs.
\cite{gz1,gz2}.
\end{remark}

For convenience, we  denote by $AP(\mathbb{T}\times D,\mathbb{E}^n)$ the set of all  functions that are almost periodic in
$t$ uniformly for $x\in D$ and denote by  $AP(\mathbb{T})$ the set of all  functions that are almost periodic in
$t\in \mathbb{T}$,  and introduce some notations:
 Let
$\alpha=\{\alpha_{n}\}$ and $\beta=\{\beta_{n}\}$ be two sequences.
Then $\beta\subset\alpha$ means that $\beta$ is a subsequence of
$\alpha$; $\alpha+\beta=\{\alpha_{n}+\beta_{n}\};
-\alpha=\{-\alpha_{n}\}$; and $\alpha$ and $\beta$ are common
subsequences of $\alpha^{'}$ and $\beta^{'}$, respectively, means
that $\alpha_{n}=\alpha^{'}_{n(k)}$ and $\beta_{n}=\beta^{'}_{n(k)}$
for some given function $n(k)$.
We  introduce the translation operator $T$,
$T_{\alpha}f(t,x)=g(t,x)$ means that
$g(t,x)=
\lim\limits_{n\rightarrow+\infty}f(t+\alpha_{n},x)$ and is
written only when the limit exists. The mode of convergence,
e.g. pointwise, uniform, etc., will be specified at each use of the
symbol.

Similar to the proofs of Theorem 3.14, Theorem 3.21 and Theorem 3.22 in \cite{li1}, respectively, one can prove the following three theorems.
\begin{theorem}\label{thm35}
Let $f\in UBC(\mathbb{T}\times D,\mathbb{E}^{n}),$ if for any
sequence $\alpha^{'}\subset\Pi$, there exists
$\alpha\subset\alpha^{'}$ such that $T_{\alpha}f$ exists
uniformly on $\widetilde{\mathbb{T}}\times S$, then $f\in AP(\mathbb{T}\times D,\mathbb{E}^n)$.
\end{theorem}

\begin{theorem}\label{thm310}
If $f\in AP(\mathbb{T}\times D,\mathbb{E}^n)$, then
for any $\varepsilon>0$, there exists a positive constant
$L=L(\varepsilon,S)$,  for any $a\in\mathbb{R}$, there exist a
constant $\eta>0$ and $\alpha\in\mathbb{R}$ such that
$\big([\alpha,\alpha+\eta]\cap\Pi\big)\subset[a,a+L]$ and
$\big([\alpha,\alpha+\eta]\cap\Pi\big)\subset
E(\varepsilon,f,S)$.
\end{theorem}

\begin{theorem}\label{thm311}
If $f,g\in AP(\mathbb{T}\times D,\mathbb{E}^n)$,  then
for any $\varepsilon>0$, $E(f,\varepsilon,S)\cap
E(g,\varepsilon,S)$ is  nonempty relatively dense.
\end{theorem}

According to Definition \ref{def32}, one can easily prove
\begin{theorem}\label{thm312}
If $f\in AP(\mathbb{T}\times D,\mathbb{E}^n)$, then
for any $\alpha\in\mathbb{R}, b\in\Pi$,  functions $\alpha
f,f(t+b,\cdot)\in AP(\mathbb{T}\times D,\mathbb{E}^n)$.
\end{theorem}

Similar to the proofs of  Theorem 3.24, Theorem 3.27, Theorem 3.28 and  Theorem 3.29 in \cite{li1}, respectively,  one can prove the following four theorems.
\begin{theorem}\label{thm313}
If $f,g\in AP(\mathbb{T}\times D,\mathbb{E}^n)$, then
$f+g,fg\in AP(\mathbb{T}\times D,\mathbb{E}^n)$, if
\,\,$\inf\limits_{t\in\mathbb{T}}|g(t,x)|>0$, then
${f}/{g}\in AP(\mathbb{T}\times D,\mathbb{E}^n)$.
\end{theorem}

\begin{theorem}
If $f_{n}\in AP(\mathbb{T}\times D,\mathbb{E}^{n}) (n=1,2,\ldots)$
  and the sequence
$\{f_{n}\}$ uniformly converges to $f$ on
$\mathbb{T}\times S$, then $f\in AP(\mathbb{T}\times D,\mathbb{E}^n)$.
\end{theorem}

\begin{theorem}\label{thm316}
If $f\in AP(\mathbb{T}\times D,\mathbb{E}^n)$, denote
$
F(t,x)=\int_{0}^{t}f(s,x)\Delta s,
$
then $F\in AP(\mathbb{T}\times D,\mathbb{E}^n)$  if
and only if $F$ is bounded on $\mathbb{T}\times S$.
\end{theorem}

\begin{theorem}\label{thm317}
If $f\in AP(\mathbb{T}\times D,\mathbb{E}^n)$,
$F(\cdot)$ is uniformly continuous on the value field of $f$,
then $F\circ f$ is almost periodic in $t$ uniformly for $x\in D$.
\end{theorem}

By Definition \ref{def32}, one can easily prove
\begin{theorem}\label{thma1}
Let $f:\mathbb{R}\rightarrow \mathbb{R}$ satisfies Lipschitz condition and $\varphi(t)\in AP(\mathbb{T})$, then $f(\varphi(t))\in AP(\mathbb{T})$.
\end{theorem}

\begin{definition}\cite{liwang} Let $A(t)$ be an $n\times n$ rd-continuous matrix on $\mathbb{T}$, the linear system
\begin{eqnarray}\label{e211}
 x^{\Delta}(t)=A(t)x(t), \quad t\in\mathbb{T}
\end{eqnarray}
is said to admit an exponential dichotomy on $\mathbb{T}$ if there exist positive constant $k$, $\alpha$, projection $P$, and the fundamental solution matrix $X(t)$ of $\eqref{e211}$, satisfying
\begin{eqnarray*}
|X(t)PX^{-1}(\sigma(s))|\leq k e_{\ominus_\alpha}(t, \sigma(s)),\,\,s,t\in
\mathbb{T},\,\, t\geq \sigma(s),
\end{eqnarray*}
\begin{eqnarray*}
|X(t)(I-P)X^{-1}(\sigma(s))|\leq
k e_{\ominus_\alpha}(\sigma(s),t),\,\,s,t\in \mathbb{T},\,\, t\leq\sigma(s),
\end{eqnarray*}
where $|\cdot|$ is a matrix norm on $\mathbb{T}$, that is, if
$A=(a_{ij})_{n\times m}$, then we can take
$|A|=\big(\sum\limits_{i=1}^n\sum\limits_{j=1}^m|a_{ij}|^2\big)^{\frac{1}{2}}$.
\end{definition}

Similar to the proof of Lemma 2.15 in \cite{liwang}, one can easily show that
\begin{lemma}\label{lem37}
Let $a_{ii}(t)$ be an uniformly bounded $rd$-continuous   function on $\mathbb{T}$, where $a_{ii}(t)>0$, $-a_{ii}(t)\in\mathcal{R}^{+}$ for every $t\in\mathbb{T}$ and
\begin{eqnarray*}
\min_{1\leq i\leq n}\{\inf_{t\in\mathbb{T}}a_{ii}(t)\}>0,
\end{eqnarray*}
then the linear system
\begin{eqnarray*}
x^{\Delta}(t)=diag(-a_{11}(t), -a_{22}(t), \ldots, -a_{nn}(t))x(t)
\end{eqnarray*}
admits an exponential dichotomy on $\mathbb{T}$.
\end{lemma}

According to Lemma \ref{lem31},
  $\widetilde{\mathbb{T}}$ is  an almost periodic time scales under Definition \ref{def21}, we denote the forward and the backward jump operators of $\widetilde{\mathbb{T}}$ by $\widetilde{\sigma}$ and $\widetilde{\rho}$, respectively.
\begin{lemma}\label{lem34}
If $t$ is a right-dense point on $\widetilde{\mathbb{T}}$, then $t$ is also a right-dense point on $\mathbb{T}$.
\end{lemma}
\begin{proof}
Let $t$ be a right-dense point on $\widetilde{\mathbb{T}}$, then
\[
t=\widetilde{\sigma}(t)=\inf\{s\in\widetilde{\mathbb{T}}:s>t\}\geq\inf\{s\in\mathbb{T}:s>t\}=\sigma(t).
\]
Since $\sigma(t)\geq t$,   $t=\sigma(t)$. The proof is complete.
\end{proof}
Similar to the proof of Lemma \ref{lem34},  one can prove the following lemma.
\begin{lemma}\label{lem35}
If $t$ is a left-dense point $\widetilde{\mathbb{T}}$, then $t$ is also a left-dense point on $\mathbb{T}$.
\end{lemma}
For each $f\in C(\mathbb{T},\mathbb{R})$, we define $\widetilde{f}:\widetilde{\mathbb{T}}\rightarrow\mathbb{R}$ by
$\widetilde{f}(t)=f(t)$ for $t\in\widetilde{\mathbb{T}}$. From Lemmas \ref{lem34} and   \ref{lem35}, we can get that
$\widetilde{f}\in C(\widetilde{\mathbb{T}},\mathbb{R})$. Therefore, $F$ defined by
\[
F(t):=\int^{t}_{t_0}\widetilde{f}(\tau)\widetilde{\Delta}\tau,\,\,
t_0,t\in\widetilde{\mathbb{T}}
\]
is an antiderivative of $f$ on $\widetilde{\mathbb{T}}$, where $\widetilde{\Delta}$ denotes the $\Delta$-derivative on $\widetilde{\mathbb{T}}$.

Set $\widetilde{\Pi}=\{\tau\in \Pi: t\pm\tau\in \widetilde{\mathbb{T}}\}$. We give our second definition of almost periodic functions on time scales as follows.

\begin{definition}\label{def32b}
Let $\mathbb{T}$ be an almost periodic time scale under sense of Definition \ref{def31}. A function  $f\in
BUC(\mathbb{T}\times D,\mathbb{E}^n)$  is  called an almost
periodic function in $t\in \mathbb{T}$ uniformly for $x\in D$ if the
$\varepsilon$-translation set of $f$
$$E\{\varepsilon,f,S\}=\{\tau\in\widetilde{\Pi}:|f(t+\tau,x)-f(t,x)|<\varepsilon,\quad\forall  (t,x)\in   \widetilde{\mathbb{T}}\times S\}$$ is  relatively
dense  for all $\varepsilon>0$ and   for each
compact subset $S$ of $D$; that is, for any given $\varepsilon>0$
and each compact subset $S$ of $D$, there exists a constant
$l(\varepsilon,S)>0$ such that each interval of length
$l(\varepsilon,S)$ contains   a $\tau(\varepsilon,S)\in
E\{\varepsilon,f,S\}$ such that
\begin{equation*}
|f(t+\tau,x)-f(t,x)|<\varepsilon, \quad\forall (t,x)\in
\widetilde{\mathbb{T}}\times S.
\end{equation*}
This   $\tau$ is called the $\varepsilon$-translation number of $f$.
\end{definition}
\begin{remark} It is clear that if a function is an almost periodic function under Definition \ref{def32}, then it is also an almost periodic function under Definition \ref{def32b}.
\end{remark}
\begin{remark} Since $\widetilde{\mathbb{T}}$   is  an almost periodic time scales under Definition \ref{def21}, under Definition \ref{def32}, all the results obtained in \cite{li1} remain valid when we restrict our discussion to $\widetilde{\mathbb{T}}$.
\end{remark}

In the following, we restrict our discuss under Definition \ref{def32b}.

Consider the following almost periodic system:
\begin{eqnarray}\label{e212}
 x^{\Delta}(t)=A(t)x(t)+f(t), \quad t\in\mathbb{T},
\end{eqnarray}
where $A(t)$ is a $n\times n$ almost periodic matrix function, $f(t)$ is a $n$-dimensional almost periodic vector function.

Similar to Lemma 2.13 in \cite{liwang}, one can easily get
\begin{lemma}\label{bs}
If  linear system $\eqref{e211}$ admits an exponential dichotomy, then system $\eqref{e212}$ has a
bounded solution $x(t)$ as follows:
\begin{eqnarray*}
x(t)=\int_{-\infty}^{t}X(t)PX^{-1}(\sigma(s))f(s)\Delta s-\int^{+\infty}_{t}X(t)(I-P)X^{-1}(\sigma(s))f(s)\Delta s,\,\, t\in \mathbb{T},
\end{eqnarray*}
where $X(t)$ is the fundamental solution matrix of $\eqref{e211}$.
\end{lemma}

By Theorem 4.19 in \cite{li1}, we have

\begin{lemma}\label{un}
Let $A(t)$ be an almost periodic matrix function and $f(t)$ be an almost periodic vector function. If \eqref{e211}
admits an exponential dichotomy, then \eqref{e212} has a unique almost periodic solution:
\[
x(t)=\int^t_{-\infty}\widetilde{X}(t)P\widetilde{X}^{-1}(\widetilde{\sigma}(s))\widetilde{f}(s)\widetilde{\Delta} s-\int_t^{+\infty}\widetilde{X}(t)(I-P)\widetilde{X}^{-1}(\widetilde{\sigma}(s))\widetilde{f}(s)\widetilde{\Delta} s,\,\, t\in \widetilde{\mathbb{T}},
\]
where $\widetilde{X}(t)$ is the restriction of the fundamental solution matrix of \eqref{e211} on $\widetilde{\mathbb{T}}$.
\end{lemma}

From Definition \ref{def32} and Lemmas \ref{bs} and \ref{un}, one can easily get the following lemma.
\begin{lemma}\label{rem34}
If  linear system $\eqref{e211}$ admits an exponential dichotomy, then system $\eqref{e212}$ has an almost periodic solution
 $x(t)$ can be expressed as:
\begin{eqnarray*}
x(t)=\int_{-\infty}^{t}X(t)PX^{-1}(\sigma(s))f(s)\Delta s-\int^{+\infty}_{t}X(t)(I-P)X^{-1}(\sigma(s))f(s)\Delta s, \,\, t\in \mathbb{T},
\end{eqnarray*}
where $X(t)$ is the fundamental solution matrix of $\eqref{e211}$.
\end{lemma}

\section{Positive almost periodic solutions for Nicholson's blowflies models}

\setcounter{equation}{0}

\indent

In this section, we will state and prove the sufficient conditions
for the existence and  exponential stability of positive almost periodic
solutions of \eqref{e1}. Throughout this section, we restrict our discussion under Definition \ref{def32b}.

Set $\mathbb{B}=\{\varphi\in C(\mathbb{T},\mathbb{R}^n):\varphi=(\varphi_1,\varphi_2,\ldots,\varphi_n)$ is
an almost periodic function on $\mathbb{T}$\} with the norm
$||\varphi||_\mathbb{B}=\max\limits_{1\leq i\leq
n}\sup\limits_{t\in\mathbb{T}}|\varphi_i(t)|$, then $\mathbb{B}$ is
a Banach space. Denote
  $\mathbb{C}=C([t_0-\theta,t_0]_\mathbb{T},\mathbb{R}^n)$ and $C\{A_1,A_2\}=\{\varphi=(\varphi_1,\varphi_2,\ldots,\varphi_n)\in \mathbb{C} : A_1\leq \varphi_i(s)\leq A_2, s\in[t_0-\theta,
t_0]_\mathbb{T},\,i=1,2,\ldots,n\}$, where $0<A_1<A_2$ are constants.

In the proofs of our results of this section, we need the following facts: There exists a unique $\varsigma\in (0,1)$ such that $\frac{1-\varsigma}{e^\varsigma}=\frac{1}{e^2} (\varsigma\approx 0.7215354)$ and  $\sup\limits_{x\geq \varsigma}\big|\frac{1-x}{e^x}\big|=\frac{1}{e^2}$. The function
$xe^{-x}$  decreases on $[1, +\infty)$.
\begin{lemma}\label{l2}
Assume that the following conditions hold.
\begin{itemize}
\item[$(H_{1})$]
$c_i, b_{ik}, \beta_{ij}, \alpha_{ij},\tau_{ij}\in AP(\mathbb{T},\mathbb{R}^+)$ and
$c_i^->0,  b_{ik}^->0, \beta_{ij}^->0, \alpha_{ij}^->0$,  $t-\tau_{ij}(t)\in \mathbb{T}$, $i,k,j=1,2,\ldots,n$.
\item[$(H_{2})$] $\sum\limits_{k=1,k\neq
i}^n\frac{b_{ik}^+}{c_{i}^-}<1$,\,$i=1,2,\ldots,n$.
\item[$(H_{3})$] There exist positive constants $A_1, A_2$ satisfy
\[
A_2> \max\limits_{1\leq i\leq n}\bigg\{\bigg[1-\sum\limits_{k=1,k\neq
i}^n\frac{b_{ik}^+}{c_{i}^-}\bigg]^{-1}\sum\limits_{j=1}^n
\frac{\beta_{ij}^+}{c_i^-\alpha_{ij}^-e}\bigg\}
\]
and
\[
\min\limits_{1\leq i\leq n}\bigg\{\bigg[1-\sum\limits_{k=1,k\neq
i}^n\frac{b_{ik}^-}{c_{i}^+}\bigg]^{-1}\sum\limits_{j=1}^n
A_2\frac{\beta_{ij}^-}{c_{i}^+}e^{-\alpha_{ij}^+A_2}\bigg\}>A_1\geq\frac{\varsigma}{\min\limits_{1\leq i,j\leq n}\{\alpha_{ij}^-\}}.
\]
\end{itemize}
 Then the solution $x(t)=(x_1(t),x_2(t),\ldots,x_n(t))$ of \eqref{e1} with the initial
value $\varphi\in C\{A_1,A_2\}$
 satisfies
\begin{eqnarray*}
A_1\leq x_i(t)\leq A_2,\,\,t\in [t_0,
+\infty)_\mathbb{T},\,i=1,2,\ldots,n.
\end{eqnarray*}
\end{lemma}
\begin{proof}Let  $x(t)=x(t;t_0,\varphi)$, where $\varphi\in C\{A_1,A_2\}$.
At first, we prove that
\begin{eqnarray}\label{e3}
x_i(t)\leq A_2,\,\,t\in [t_0, \eta(\varphi))_\mathbb{T},\,i=1,2,\ldots,n,
\end{eqnarray}
where $[t_0, \eta(\varphi))_\mathbb{T}$ is the maximal
right-interval of existence of $x(t;t_0,\varphi)$.
To prove this claim, we show that for any $p>1$, the following
inequality holds
\begin{eqnarray}\label{e4}
x_i(t)<pA_2,\,\, t\in [t_0, \eta(\varphi))_\mathbb{T},\,i=1,2,\ldots,n.
\end{eqnarray}
By way of contradiction, assume that \eqref{e4} does not hold. Then,
there exists $i_0\in\{1,2,\dots,n\}$ and the first time $t_1\in[t_0,
\eta(\varphi))_\mathbb{T}$ such that
\[
x_{i_0}(t_1)\geq pA_2, \,\,x_{i_0}(t)<p A_2,\,\,
t\in[t_0-\theta,t_1)_{\mathbb{T}},
\]
\[
x_{k}(t)<p A_2,\,\, \mathrm{for} \, k\neq i_0,
\,\,t\in[t_0-\theta,t_1]_{\mathbb{T}},\,k=1,2,\dots,n.
\]
Therefore, there must be a positive constant $a\geq1$ such that
\[
x_{i_0}(t_1)=apA_2, \,\,x_{i_0}(t)<apA_2,\,\,
t\in[t_0-\theta,t_1)_{\mathbb{T}},
\]
\[
x_{k}(t)<apA_2,\,\, \mathrm{for} \, k\neq i_0,
\,\,t\in[t_0-\theta,t_1]_{\mathbb{T}},\,k=1,2,\dots,n.
\]
In view of the fact that $\sup\limits_{u\geq 0}ue^{-u}=\frac{1}{e}$ and $ap>1$,
we can obtain
\begin{eqnarray*}\label{c1}
0\leq x_{i_0}^{\Delta}(t_1)
&=&-c_{i_0}(t_1)x_{i_0}(t_1)+\sum\limits_{k=1,k\neq
i_0}^nb_{i_0k}(t_1)x_k(t_1)\nonumber\\
&&+\sum\limits_{j=1}^n
\frac{\beta_{i_0j}(t_1)}{\alpha_{i_0j}(t_1)}\alpha_{i_0j}(t_1)x_{i_0}(t_1-\tau_{i_0j}(t_1))
e^{-\alpha_{i_0j}(t_0)x_{i_0}(t_0-\tau_{i_0j}(t_0))}\nonumber\\
&\leq&-c_{i_0}^-apA_2+\sum\limits_{k=1,k\neq
i_0}^nb_{i_0k}^+apA_2+\sum\limits_{j=1}^n
\frac{\beta_{i_0j}^+}{\alpha_{i_0j}^-}\cdot\frac{1}{e}\nonumber\\
&\leq& apc_{i_0}^-\Big(-A_2+\sum\limits_{k=1,k\neq
i_0}^n\frac{A_2b_{i_0k}^+}{c_{i_0}^-}+\sum\limits_{j=1}^n
\frac{\beta_{i_0j}^+}{c_{i_0}^-\alpha_{i_0j}^-e}\Big)<0,
\end{eqnarray*}
which  is a contradiction and hence \eqref{e4} holds. Let
$p\rightarrow 1$, we have that \eqref{e3} is true. Next, we show
that
\begin{eqnarray}\label{e31}
x_i(t)\geq A_1,\,\,t\in [t_0, \eta(\varphi))_\mathbb{T},\,i=1,2,\ldots,n.
\end{eqnarray}
To prove this claim, we show that for any $l<1$, the following
inequality holds
\begin{eqnarray}\label{e41}
x_i(t)>lA_1,\,\,t\in [t_0, \eta(\varphi))_\mathbb{T},\,i=1,2,\ldots,n.
\end{eqnarray}
By way of contradiction, assume that \eqref{e41} does not hold.
Then, there exists $i_1\in\{1,2,\dots,n\}$ and the first time
$t_2\in[t_0, \eta(\varphi))_\mathbb{T}$ such that
\[
x_{i_1}(t_2)\leq lA_1, \,\,x_{i_1}(t)>l,\,\,
t\in[t_0-\theta,t_2)_{\mathbb{T}},
\]
\[
x_{k}(t)>lA_1,\,\, \mathrm{for} \, k\neq i_1,
\,\,t\in[t_0-\theta,t_2]_{\mathbb{T}},\,k=1,2,\dots,n.
\]
Therefore, there must be a positive constant $c\leq1$ such that
\[
x_{i_1}(t_2)=clA_1, \,\,x_{i_1}(t)>cl,\,\,
t\in[t_0-\theta,t_2)_{\mathbb{T}},
\]
\[
x_{k}(t)>clA_1,\,\, \mathrm{for} \, k\neq i_1,
\,\,t\in[t_0-\theta,t_2]_{\mathbb{T}},\,k=1,2,\dots,n.
\]
Noticing that $cl<1$, it follows that
\begin{eqnarray*}\label{c2}
0\geq x_{i_1}^{\Delta}(t_2)
&=&-c_{i_1}(t_2)x_{i_1}(t_2)+\sum\limits_{k=1,k\neq
i_1}^nb_{i_1k}(t_2)x_k(t_2)\nonumber\\
&& +\sum\limits_{j=1}^n
\beta_{i_1j}(t_2)x_{i_1}(t_2-\tau_{i_1j}(t_2))
e^{-\alpha_{i_1j}(t_2)x_{i_1}(t_2-\tau_{i_1j}(t_2))}\nonumber\\
&\geq&-c_{i_1}^+clA_1+\sum\limits_{k=1,k\neq
i_1}^nb_{i_1k}^-clA_1+\sum\limits_{j=1}^n
A_2\frac{\alpha_{i_1j}^+\beta_{i_1j}^-}{\alpha_{i_1j}^+}e^{-\alpha_{i_1j}^+A_2}\nonumber\\
&=& cl c_{i_1}^+\Big(-A_1+\sum\limits_{k=1,k\neq
i_1}^nA_1\frac{b_{i_1k}^-}{c_{i_1}^+}+\sum\limits_{j=1}^n
A_2\frac{\beta_{i_1j}^-}{c_{i_1}^+}e^{-\alpha_{i_1j}^+A_2}\Big)>0,
\end{eqnarray*}
which is a contradiction and hence \eqref{e41} holds. Let
$l\rightarrow 1$, we have that \eqref{e31} is true. Similar to the proof of Theorem 2.3.1 in \cite{jack}, we easily obtain $\eta(\varphi)=+\infty$. This completes
the proof.
\end{proof}
\begin{remark}\label{mu}
If $\mathbb{T}=\mathbb{R}$, then $\mu(t)\equiv 0$, so, $-c_i\in\mathcal{R}^+$. If $\mathbb{T}=\mathbb{Z}$, then $\mu(t)\equiv 1$, so, $-c_i\in\mathcal{R}^+$ if and only if $c_i<1$.
\end{remark}

\begin{theorem}\label{th31}
Assume that $(H_1)$  and $(H_3)$ hold. Suppose further that
\begin{itemize}
  \item [$(H_4)$]
   $-c_i\in
\mathcal{R}^+$, where $\mathcal{R}^+$ denotes the set of positive
regressive functions, $i=1,2,\ldots,n$.
 \item [$(H_5)$] $\sum\limits_{k=1,k\neq i}^nb_{ik}^++
\sum\limits^{n}_{j=1}\frac{\beta_{ij}^+}{e^2}<c_i^-, i=1,2,\ldots,n.$
\end{itemize}
Then system \eqref{e1} has a  positive almost
periodic solution
in the region $\mathbb{B}^*=\{\varphi|\,\,\varphi\in \mathbb{B}, A_1 \leq
\varphi_i(t)\leq A_2, t\in\mathbb{T},i=1,2,\ldots,n\}$.
\end{theorem}
\begin{proof}
For any given $\varphi\in\mathbb{B}$, we consider the following almost periodic dynamic system:
\begin{eqnarray}\label{e311}
x_i^\Delta(t)&=&-c_i(t)x_i(t)+\sum\limits_{k=1,k\neq
i}^nb_{ik}(t)\varphi_k(t)\nonumber\\
&&+\sum\limits_{j=1}^n\beta_{ij}(t)\varphi_i(t-\tau_{ij}(t))e^{-\alpha_{ij}(t)\varphi_i(t-\tau_{ij}(t))},\,\,
i=1,2,\ldots,n.
\end{eqnarray}

Since $\min_{1\leq i\leq n}\{ c_i^-\}>0$, $t\in \mathbb{T}$, it follows from Lemma \ref{lem37} that the
linear system
\begin{eqnarray*}\label{e312}
x_i^\Delta(t)=-c_i(t)x_i(t),\,\,
i=1,2,\ldots,n
\end{eqnarray*}
admits an exponential dichotomy on $\mathbb{T}$. Thus, by Lemma \ref{rem34}, we obtain that system
\eqref{e311} has an almost periodic solution $x_\varphi=(x_{\varphi_1},x_{\varphi_2},\ldots,x_{\varphi_n})$, where
\begin{eqnarray*}
{x_{\varphi}}_i(t)&=&\int_{-\infty}^te_{-c_i}(t,\sigma(s))\bigg[\sum\limits_{k=1,k\neq
i}^nb_{ik}(s)\varphi_k(s)\nonumber\\
&&+\sum\limits_{j=1}^n\beta_{ij}(s)\varphi_i(s-\tau_{ij}(s))e^{-\alpha_{ij}(s)\varphi_i(s-\tau_{ij}(s))}\bigg]\Delta
s,\,\,i=1,2,\ldots,n.
\end{eqnarray*}

Define a mapping $T:\mathbb{B}^*\rightarrow\mathbb{B}^*$ by
\begin{eqnarray*}
T\varphi(t)=x_{\varphi}(t),\,\,\forall \varphi\in\mathbb{B}^*.
\end{eqnarray*}
Obviously, $\mathbb{B}^*=\{\varphi|\,\,\varphi\in \mathbb{B}, A_1 \leq
\varphi_i(t)\leq A_2, t\in\mathbb{T},i=1,2,\ldots,n\}$ is a closed  subset of $\mathbb{B}$. For any $\varphi\in\mathbb{B}^*$, by use of $(H_2)$, we have
\begin{eqnarray*}
{x_{\varphi}}_i(t)&\leq&\int_{-\infty}^te_{-c_i^-}(t,\sigma(s))\bigg[\sum\limits_{k=1,k\neq
i}^nb_{ik}^+A_2+\sum\limits_{j=1}^n\frac{\beta_{ij}^+}{\alpha_{ij}^-}\times\frac{1}{e}\bigg]\Delta s\\
&\leq& \frac{1}{c_i^-}\bigg[\sum\limits_{k=1,k\neq
i}^nb_{ik}^+A_2+\sum\limits_{j=1}^n\frac{\beta_{ij}^+}{\alpha_{ij}^-}\times\frac{1}{e}\bigg]\\
&\leq&A_2,\,\,i=1,2,\ldots,n
\end{eqnarray*}
and we also have
\begin{eqnarray*}
{x_{\varphi}}_i(t)&\geq&\int_{-\infty}^te_{-c_i^+}(t,\sigma(s))\bigg[\sum\limits_{k=1,k\neq
i}^nA_1b_{ik}^-+\sum\limits_{j=1}^n {\beta_{ij}^-} \varphi_i(s-\tau_{ij}(s))e^{-\alpha_{ij}^+\varphi_i(s-\tau_{ij}(s))}\bigg]\Delta s\\
&\geq& \frac{1}{c_i^+}\bigg[\sum\limits_{k=1,k\neq
i}^nA_1b_{ik}^-+\sum\limits_{j=1}^nA_2\beta_{ij}^-e^{-\alpha_{ij}^+A_2}\bigg]\\
&\geq&A_1,\,\,i=1,2,\ldots,n.
\end{eqnarray*}
Therefore, the mapping $T$ is a self-mapping from $\mathbb{B}^*$ to $\mathbb{B}^*$.

Next, we prove that the mapping $T$ is a contraction mapping on $\mathbb{B}^*$. Since
$\sup\limits_{u\geq \varsigma}|\frac{1-u}{e^u}|=\frac{1}{e^2}$, we find that
\begin{eqnarray*}
|xe^{-x}-ye^{-y}|&=&\Big|\frac{1-(x+\xi(y-x))}{e^{x+\xi(y-x)}}\Big||x-y|\\
&\leq&\frac{1}{e^2}|x-y|,\,x,y\geq \varsigma,\,0<\xi<1.
\end{eqnarray*}For any $\varphi=(\varphi_1,
\varphi_2, \ldots, \varphi_n)^T$, $\psi=(\psi_1, \psi_2, \ldots,
\psi_n)^T \in \mathbb{B}^*$, we obtain that
\begin{eqnarray*}
&&|(T\varphi)_i(t)-(T\psi)_i(t)|\nonumber\\
&\leq&\Big|\int_{-\infty}^t\mathrm{e}_{-c_i}(t,\sigma(s))\sum\limits_{k=1,k\neq
i}^nb_{ik}(s)\big(\varphi_k(s)-\psi_k(s)\big)\Delta
s\Big|\nonumber\\
&&+\bigg|\int_{-\infty}^t\mathrm{e}_{-c_i}(t,\sigma(s))
\sum\limits^{n}_{j=1}\beta_{ij}(s)\bigg(\varphi_i(s-\tau_{ij}(s))e^{-\alpha_{ij}(s)\varphi_i(s-\tau_{ij}(s))}
\nonumber\\
&&-\psi_i(s-\tau_{ij}(s))e^{-\alpha_{ij}(s)\psi_i(s-\tau_{ij}(s))}\bigg)\Delta
s\bigg|\nonumber\\
&\leq&\frac{1}{c_i^-}\sum\limits_{k=1,k\neq
i}^nb_{ik}^+\|\varphi-\psi\|_{\mathbb{B}}+\bigg|\int_{-\infty}^t\mathrm{e}_{-c_i}(t,\sigma(s))
\sum\limits^{n}_{j=1}\frac{\beta_{ij}(s)}{\alpha_{ij}(s)}\Big(\alpha_{ij}(s)\varphi_i(s-\tau_{ij}(s))\nonumber\\
&&\times e^{-\alpha_{ij}(s)\varphi_i(s-\tau_{ij}(s))}
-\alpha_{ij}(s)\psi_i(s-\tau_{ij}(s))e^{-\alpha_{ij}(s)\psi_i(s-\tau_{ij}(s))}\Big)\Delta
s\bigg|\nonumber\\
&\leq&\bigg(\frac{1}{c_i^-}\sum\limits_{k=1,k\neq i}^nb_{ik}^++
\sum\limits^{n}_{j=1}\frac{\beta_{ij}^+}{c_i^-e^2}\bigg)\|\varphi-\psi\|_{\mathbb{B}},\,\,i=1,2,\ldots,n.
\end{eqnarray*}
It follows that
\[
\|T\phi-T\psi\|_\mathbb{B}<\|\varphi-\psi\|_{\mathbb{B}},
\]
which implies that $T$ is a contraction. By the fixed point
theorem in Banach space, $T$ has a unique fixed point $\varphi^*\in\mathbb{B}^*$ such that $T\varphi^*=\varphi^*$. In view of \eqref{e311}, we see that $\varphi^*$  is a solution of \eqref{e1}. Therefore, \eqref{e1} has a positive almost periodic solution in the region $\mathbb{B}^*$. This completes the
proof.
\end{proof}

\begin{definition}
Let $x^*(t)=(x_1^*(t), x_2^*(t), \ldots, x_n^*(t))^T$ be an almost
periodic solution of \eqref{e1} with initial value
$\varphi^*(s)=(\varphi_1^*(s), \varphi_2^*(s), \ldots,
\varphi_n^*(s))^T\in C\{A_1,A_2\}$. If there exist positive constants $\lambda$ with
$\ominus\lambda\in \mathcal{R}^{+}$ and $M>1$ such that such that
for an arbitrary solution $x(t)=(x_1(t), x_2(t),
\ldots, x_n(t))^T$
of \eqref{e1} with initial value $\varphi(s)=(\varphi_1(s),
\varphi_2(s), \ldots, \varphi_n(s))^T\in C\{A_1,A_2\}$ satisfies
\[
||x-x^*||\leq M||\varphi-\varphi^*||e_{\ominus\lambda}(t,t_0),\,\,
t_0\in [-\theta,\infty)_{\mathbb{T}},\,t\geq t_0,
\]
where $||\varphi-\varphi^*||_0=\max\limits_{1\leq i\leq n}\bigg\{\sup\limits_{t\in [t_0-\theta,t_0]}|\varphi_i(t)-\varphi_i^*(t)|\bigg\}$ for $\varphi,\psi\in C\{A_1,A_2\}$.
Then the solution $x^*(t)$ is said to be  exponentially stable.
\end{definition}

\begin{theorem}\label{th32}
Assume that $(H_1)$, $(H_3)$-$(H_5)$ hold. Then the positive almost periodic solution $x^*(t)$ in the region $\mathbb{B}^*$ of \eqref{e1} is unique and  exponentially stable.
\end{theorem}
\begin{proof}
By Theorem \ref{th31}, \eqref{e1} has a positive almost periodic solution
$x_i^*(t)$ in the region $\mathbb{B}^*$.  Let
$x(t)=(x_1(t),x_2(t),\ldots, x_n(t))^T$ be any arbitrary solution of
\eqref{e1} with initial value
$\varphi(s)=(\varphi_1(s),\varphi_2(s),\ldots,\varphi_n(s))^T\in C\{A_1,A_2\}$. Then
it follows from \eqref{e1} that for $t\geq t_0, i=1,2,\ldots,n$,
\begin{eqnarray}\label{e5}
&&(x_i(t)-x_i^*(t))^\Delta\nonumber\\
&=&-c_i(t)(x_i(t)-x_i^*(t))+\sum\limits_{k=1,k\neq
i}^nb_{ik}(t)(x_k(t)-x_k^*(t))\nonumber\\
&&+\sum\limits_{j=1}^n\beta_{ij}(t)\big[x_i(t-\tau_{ij}(t))e^{-\alpha_{ij}(t)x_i(t-\tau_{ij}(t))}
-x_i^*(t-\tau_{ij}(t))e^{-\alpha_{ij}(t)x_i^*(t-\tau_{ij}(t))}\big].
\end{eqnarray}
The initial condition of \eqref{e5} is
\[
\psi_i(s)=\varphi_i(s)-x_i^*(s),\,\,s\in [t_0-\theta,
t_0]_{\mathbb{T}},\,i=1,2,\ldots,n.
\]
For convenience, we denote $u_i(t)=x_i(t)-x_i^*(t), i=1,2,\ldots,n$.
Then, by \eqref{e5},  we have
\begin{eqnarray}\label{e42}
u_i(t)&=&u_i(t_0)e_{-c_i}(t,t_0)+\int_{t_0}^{t}e_{-c_{i}}(t,\sigma(s))\sum\limits_{k=1,k\neq
i}^nb_{ik}(s)u_k(s)\Delta
s\nonumber\\
&&
+\int_{t_0}^{t}e_{-c_{i}}(t,\sigma(s))\sum\limits_{j=1}^n\beta_{ij}(s)\bigg[x_i(s-\tau_{ij}(s))e^{-\alpha_{ij}(s)x_i(s-\tau_{ij}(s))}
\nonumber\\
&&-x_i^*(s-\tau_{ij}(s))e^{-\alpha_{ij}(s)x_i^*(s-\tau_{ij}(s))}\bigg]\Delta
s, \,\,t\geq t_0, i=1,2,\ldots,n.
\end{eqnarray}
For $\omega\in \mathbb{R}$, let $\Gamma_{i}(\omega)$ be defined by
{\setlength\arraycolsep{2pt}\begin{eqnarray*}
\Gamma_{i}(\omega)&=&c_{i}^{-}-\omega-\exp\{\omega\sup\limits_{s\in
\mathbb{T}}\mu(s)\}\bigg(\sum\limits_{k=1,k\neq
i}^nb_{{i}k}^++\frac{1}{e^2}\sum\limits^{n}_{j=1}\beta_{ij}^+\exp\{\omega\tau_{ij}^+\}\bigg),\,i=1,2,\ldots,n.
\end{eqnarray*}}
In view of $(H_{2})$,   we have that
{\setlength\arraycolsep{2pt}\begin{eqnarray*}
\Gamma_{i}(0)=c_{i}^{-}-\bigg(\sum\limits_{k=1,k\neq
i}^nb_{{i}k}^++\frac{1}{e^2}\sum\limits^{n}_{j=1}\beta_{ij}^+\bigg)>
0, i=1,2,\ldots,n.
\end{eqnarray*}}
Since $\Gamma_{i}(\omega)$ is continuous on $[0,+\infty)$ and
$\Gamma_{i}(\omega)\rightarrow -\infty$ as $\omega\rightarrow
+\infty$, so there exists $\omega_{i}> 0$ such that
$\Gamma_{i}(\omega_{i})=0$ and $\Gamma_{i}(\omega)> 0$ for
$\omega\in(0,\omega_{i}), i=1,2,\ldots,n$. By choosing a positive
constant
$a=\min\big\{\omega_{1},\omega_{2},\ldots,\omega_{n}\big\}$, we have
$\Gamma_{i}(a)\geq 0, i=1,2,\ldots,n.$ Hence, we can choose a
positive constant $0< \alpha < \min\big\{a,\min\limits_{1\leq i \leq
n}\{c_{i}^{-}\}\big\}$ such that
\[
\Gamma_{i}(\alpha)>0,\,\,i=1,2,\ldots,n,
\]
which implies that {\setlength\arraycolsep{2pt}\begin{eqnarray*}
&&\frac{\exp\{\alpha\sup\limits_{s\in
\mathbb{T}}\mu(s)\}}{c_{i}^{-}-\alpha}\bigg(\sum\limits_{k=1,k\neq
i}^nb_{{i}k}^++\frac{1}{e^2}\sum\limits^{n}_{j=1}\beta_{ij}^+\exp\{\alpha\tau_{ij}^+\}\bigg)<
1,\,\,i=1,2,\ldots,n.
\end{eqnarray*}}
Take
\[
M=\max\limits_{1\leq i\leq
n}\bigg\{\frac{c_{i}^-}{\sum\limits_{k=1,k\neq
i}^nb_{{i}k}^++\frac{1}{e^2}\sum\limits^{n}_{j=1}\beta_{ij}^+}\bigg\}.
\]
It follows from $(H_5)$ that $M>1$. Besides, we can obtain that
{\setlength\arraycolsep{2pt}\begin{eqnarray*} \frac{1}{M}&<&
\frac{\exp\{\alpha\sup\limits_{s\in
\mathbb{T}}\mu(s)\}}{c_{i}^{-}-\alpha}\bigg(\sum\limits_{k=1,k\neq
i}^nb_{{i}k}^++\frac{1}{e^2}\sum\limits^{n}_{j=1}\beta_{ij}^+\exp\{\alpha\tau_{ij}^+\}\bigg).
\end{eqnarray*}}
In addition, noticing  that $e_{\ominus\alpha}(t,t_0)\geq1$ for
$t\in[t_0-\theta,t_0]_{\mathbb{T}}$. Hence, it is obvious that
\[
||u||_{\mathbb{B}}\leq M
e_{\ominus\alpha}(t,t_0)\|\psi\|_0,\,\forall\,
t\in[t_0-\theta,t_0]_{\mathbb{T}}.
\]
We claim that
\begin{equation}\label{e43}
||u||_{\mathbb{B}}\leq M
e_{\ominus\alpha}(t,t_0)\|\psi\|_{0},\,\,\,\,\,\forall\,
t\in(t_0,+\infty)_{\mathbb{T}}.
\end{equation}
To prove this claim, we show that for any $p>1$, the following
inequality holds
\begin{equation}\label{e44}
||u||_{\mathbb{B}}< pM
e_{\ominus\alpha}(t,t_0)\|\psi\|_0,\,\,\,\,\,\forall\,
t\in(t_0,+\infty)_{\mathbb{T}},
\end{equation}
which implies that, for $i=1,2,\ldots,n$, we have
\begin{equation}\label{e45}
|u_i(t)|< pM
e_{\ominus\alpha}(t,t_0)\|\psi\|_{0},\,\forall
t\in(t_0,+\infty)_{\mathbb{T}}.
\end{equation}
By way of contradiction, assume that \eqref{e45} is not true. Then
there exists $t_1\in(t_0,+\infty)_{\mathbb{T}}$ and
$i_0\in\{1,2,\ldots,n\}$ such that
\[
|u_{i_0}(t_1)|\geq pM
e_{\ominus\alpha}(t_1,t_0)\|\psi\|_{0},
\,\,|u_{i_0}(t)|<pM
e_{\ominus\alpha}(t,t_0)\|\psi\|_{0},\,\,
t\in(t_0,t_1)_{\mathbb{T}},
\]
\[
|u_{k}(t)|\leq pM
e_{\ominus\alpha}(t,t_0)\|\psi\|_{0},\,\,\mathrm{for} \,
k\neq i_0,\, \,t\in(t_0,t_1]_{\mathbb{T}}, k=1,2,\dots,n.
\]
Therefore, there must be a constant $q\geq1$ such that
\[
|u_{i_0}(t_1)|= qpM
e_{\ominus\alpha}(t_1,t_0)\|\psi\|_{0},
\,\,|u_{i_0}(t)|<qpM
e_{\ominus\alpha}(t,t_0)\|\psi\|_{0},\,\,
t\in(t_0,t_1)_{\mathbb{T}},
\]
\[
|u_{k}(t)|<qpM
e_{\ominus\alpha}(t_1,t_0)\|\psi\|_{0},\,\,\mathrm{for}
\, k\neq i_0,\, \,t\in(t_0,t_1]_{\mathbb{T}}, k=1,2,\dots,n.
\]
According to \eqref{e42}, we have
\begin{eqnarray*}
|u_{i_0}(t_1)|&=&\bigg|u_{i_0}(t_0)e_{-c_{i_0}}(t_1,t_0)+\int_{t_0}^{t_1}e_{-c_{{i_0}}}(t_1,\sigma(s))\sum\limits_{k=1,k\neq
i_0}^nb_{{i_0}k}(s)u_k(s)\Delta
s\nonumber\\
&&
+\int_{t_0}^{t_1}e_{-c_{{i_0}}}(t_1,\sigma(s))
\sum\limits_{j=1}^n\beta_{{i_0}j}(s)\big[x_{i_0}(s-\tau_{{i_0}j}(s))e^{-\alpha_{{i_0}j}(s)x_{i_0}(s-\tau_{{i_0}j}(s))}
\nonumber\\
&&-x_{i_0}^*(s-\tau_{{i_0}j}(s))e^{-\alpha_{{i_0}j}(s)x_{i_0}^*(s-\tau_{{i_0}j}(s))}\big]\Delta
s\bigg|\nonumber\\
&\leq& e_{-c_{i_0}}(t_1,t_0)\|\psi\|_{0}+qpM
e_{\ominus\alpha}(t_1,t_0)\|\psi\|_{0}\nonumber\\
&&\times\int_{t_0}^{t_1}e_{-c_{i_0}}(t_1,\sigma(s))e_{\alpha}(t_1,\sigma(s))
\bigg(\sum\limits_{k=1,k\neq
i_0}^nb_{{i_0}k}^+e_{\alpha}(\sigma(s),s)\nonumber\\&&+
\sum\limits^{m}_{j=1}\frac{\beta_{i_0j}^+}{e^2}e_{\alpha}(\sigma(s),s-\tau_{i_0j}(s))\bigg)\Delta
s\nonumber\\
&\leq& e_{-c_{i_0}}(t_1,t_0)\|\psi\|_{0}+qpM
e_{\ominus\alpha}(t_1,t_0)\|\psi\|_{0}\nonumber\\
&&\times\int_{t_0}^{t_1}e_{-c_{i_0}\oplus\alpha}(t_1,\sigma(s))
\bigg(\sum\limits_{k=1,k\neq
i}^nb_{{i_0}k}^+\exp\{\alpha\sup\limits_{s\in
\mathbb{T}}\mu(s)\}\nonumber\\&&+
\sum\limits^{m}_{j=1}\frac{\beta_{i_0j}^+}{e^2}\exp\{\alpha(\tau_{i_0j}^++\sup\limits_{s\in
\mathbb{T}}\mu(s))\}\bigg)\Delta s
\nonumber\\
&=& e_{-c_{i_0}}(t_1,t_0)\|\psi\|_{0}+qpM
e_{\ominus\alpha}(t_1,t_0)\|\psi\|_{0}\exp\{\alpha\sup\limits_{s\in
\mathbb{T}}\mu(s)\}\bigg(\sum\limits_{k=1,k\neq
i_0}^nb_{{i_0}k}^+\nonumber\\&&+
\sum\limits^{m}_{j=1}\frac{\beta_{i_0j}^+}{e^2}\exp\{\alpha\tau_{i_0j}^+\}\bigg)
\int_{t_0}^{t_1}e_{-c_{i_0}\oplus\alpha}(t_1,\sigma(s)) \Delta s
\nonumber\\
&=&qpM
e_{\ominus\alpha}(t_1,t_0)\|\psi\|_{0}\bigg\{\frac{1}{qpM}e_{-c_{i_0}\oplus\alpha}(t_1,t_0)
+\exp\{\alpha\sup\limits_{s\in
\mathbb{T}}\mu(s)\}\bigg(\sum\limits_{k=1,k\neq
i_0}^nb_{{i_0}k}^+\nonumber\\&&+
\sum\limits^{m}_{j=1}\frac{\beta_{i_0j}^+}{e^2}\exp\{\alpha\tau_{i_0j}^+\}\bigg)
\int_{t_0}^{t_1}e_{-c_{i_0}\oplus\alpha}(t_1,\sigma(s))\Delta
s\bigg\}
\nonumber\\
&<&qpM
e_{\ominus\alpha}(t_1,t_0)\|\psi\|_{0}\bigg\{\frac{1}{qpM}e_{-(c_{i_0}^{-}-\alpha)}(t_1,t_0)
+\exp\{\alpha\sup\limits_{s\in
\mathbb{T}}\mu(s)\}\bigg(\sum\limits_{k=1,k\neq
i_0}^nb_{{i_0}k}^+\nonumber\\&&+
\sum\limits^{m}_{j=1}\frac{\beta_{i_0j}^+}{e^2}\exp\{\alpha\tau_{i_0j}^+\}\bigg)
\frac{1}{-(c_{i_0}^{-}-\alpha)}\int_{t_0}^{t_1}\big(-(c_{i_0}^{-}-\alpha)\big)e_{-(c_{i_0}^{-}-\alpha)}(t_1,\sigma(s))\Delta
s\bigg\}
\nonumber\\
&\leq&qpM
e_{\ominus\alpha}(t_1,t_0)\|\psi\|_{0}\Bigg\{\bigg[\frac{1}{qpM}-\frac{\exp\{\alpha\sup\limits_{s\in
\mathbb{T}}\mu(s)\}}{c_{i_0}^{-}-\alpha}\Big(\sum\limits_{k=1,k\neq
i_0}^nb_{{i_0}k}^+\nonumber\\&&
+\frac{1}{e^2}\sum\limits^{n}_{j=1}\beta_{ij}^+
\exp\{\alpha\tau_{i_0j}^+\}\Big)\bigg]e_{-(c_{i_0}^{-}-\alpha)}(t_1,t_0)
+\frac{\exp\{\alpha\sup\limits_{s\in
\mathbb{T}}\mu(s)\}}{c_{i_0}^{-}-\alpha}\Big(\sum\limits_{k=1,k\neq
i_0}^nb_{{i_0}k}^+\nonumber\\&&+\frac{1}{e^2}\sum\limits^{n}_{j=1}\beta_{i_0j}^+
\exp\{\alpha\tau_{i_0j}^+\}\Big)\Bigg\}
\nonumber\\
&<&qpM
e_{\ominus\alpha}(t_1,t_0)\|\psi\|_{0}\Bigg\{\bigg[\frac{1}{M}-\frac{\exp\{\alpha\sup\limits_{s\in
\mathbb{T}}\mu(s)\}}{c_{i_0}^{-}-\alpha}\Big(\sum\limits_{k=1,k\neq
i_0}^nb_{{i_0}k}^+\nonumber\\&&
+\frac{1}{e^2}\sum\limits^{n}_{j=1}\beta_{i_0j}^+
\exp\{\alpha\tau_{i_0j}^+\}\Big)\bigg]e_{-(c_{i_0}^{-}-\alpha)}(t_1,t_0)
+\frac{\exp\{\alpha\sup\limits_{s\in
\mathbb{T}}\mu(s)\}}{c_{i_0}^{-}-\alpha}\Big(\sum\limits_{k=1,k\neq
i_0}^nb_{{i_0}k}^+\nonumber\\
&&+\frac{1}{e^2}\sum\limits^{n}_{j=1}\beta_{i_0j}^+
\exp\{\alpha\tau_{i_0j}^+\}\Big)\Bigg\}
\nonumber\\
&<&qpM e_{\ominus\alpha}(t_1,t_0)\|\psi\|_{0},
\end{eqnarray*}
which is a contradiction. Therefore, \eqref{e45} and \eqref{e44}
hold. Let $p\rightarrow 1$, then \eqref{e43} holds. Hence, we have
that
\[||u||_{\mathbb{B}}\leq
M\|\psi\|_{0}e_{\ominus\alpha}(t,t_0),\,\, t\in
[t_0,+\infty)_{\mathbb{T}},
\]
which implies that the positive almost periodic solution $x^*(t)$ of
\eqref{e1} is  exponentially stable. The  exponential stability of $x^*(t)$ implies that the uniqueness of the positive almost periodic solution. The proof is complete.
\end{proof}

\begin{remark}
It is easy to see that under definitions of almost periodic time scales and  almost periodic functions in \cite{li1}, the conclusions of Theorems \ref{th31} and \ref{th32} are true.
\end{remark}
\begin{remark}
From Remark \ref{mu}, Theorem \ref{th31} and Theorem \ref{th32}, we can easily see that if $c_i(t)<1, i=1,2,\ldots,n$, then the
continuous-time Nicholson's blowflies models and the discrete-time
analogue have the same dynamical behaviors. This fact provides a theoretical basis for the numerical simulation of continuous-time
  Nicholson's blowflies models.
\end{remark}
\begin{remark}
Our results and methods of this paper are different from those in \cite{17}.
\end{remark}

\begin{remark}
When $\mathbb{T}=\mathbb{R}$ or $\mathbb{T}=\mathbb{Z}$, our results  of this section are also new. If we take $\mathbb{T}=\mathbb{R}, A_1=1, A_2=e$, then Lemma \ref{l2}, Theorem \ref{th31}  and Theorem \ref{th31} improve Lemma 2.4, Theorem 3.1 and Theorem 3.2 in \cite{8}, respectively.
\end{remark}

\section{ An example}
 \setcounter{equation}{0}
 \indent

In this section, we present an example to illustrate the feasibility
of our results obtained in previous sections.

\begin{example} In system \eqref{e1},
let $n=3$ and take coefficients as follows:
\[
c_1(t)=0.21+0.01\sin\bigg(\frac{1}{3}t\bigg),\,b_{12}(t)=0.03+0.01\cos \pi
t,\,b_{13}(t)=0.06+0.01\cos\sqrt{3} t,
\]
\[
\beta_{11}(t)=0.07+0.02\sin \pi
t,\,\beta_{12}(t)=0.15+0.01\cos\sqrt{3}t,\]
\[\beta_{13}(t)=0.15+0.01\sin
\bigg(\frac{5}{6}t\bigg),\,
\alpha_{11}(t)=\alpha_{12}(t)=\alpha_{13}(t)=0.91+0.09|\sin\sqrt{3}t|,
\]
\[\tau_{11}(t)=e^{0.2|\sin\pi
t|},\,\tau_{12}(t)=e^{0.4|\cos(\pi t+\frac{\pi}{2}) |},\tau_{13}(t)=e^{0.5|\sin\pi t|},
\]
\[
c_2(t)=0.3+0.02\sin \bigg(\frac{4}{3}t\bigg), b_{21}(t)=0.05+0.01\cos \sqrt{3}
t,\,b_{23}(t)=0.05+0.01\sin \sqrt{2} t,\]
\[\beta_{21}(t)=0.06+0.01\cos\pi t,\,
\beta_{22}(t)=0.04+0.01\cos\sqrt{3}t,\,\beta_{23}(t)=0.09+0.01\cos
\bigg(\frac{1}{3}t\bigg),\]
\[\alpha_{21}(t)=0.8+0.2\sin \sqrt{2}t,
\alpha_{22}(t)=0.8+0.2\cos\sqrt{2}t, \alpha_{23}(t)=0.8+0.2\sin
\pi t,\] \[\tau_{21}(t)=e^{0.2|\cos(\pi t+\frac{\pi}{2})|}, \tau_{22}(t)=e^{0.3|\sin
3\pi t|},
 \tau_{23}(t)=e^{0.1|\cos( 2\pi t+\frac{\pi}{2})|},
 \]
 \[
c_3(t)=0.41+0.01\sin \bigg(\frac{1}{3}t\bigg),\, b_{31}(t)=0.16+0.01\sin
\sqrt{3}t,\,b_{32}(t)=0.13+0.01\cos \sqrt{2} t,
\]
 \[
\beta_{31}(t)=0.02+0.01\cos \bigg(\frac{1}{6}t\bigg),\,
\beta_{32}(t)=0.032+0.01\cos\sqrt{2}t,
\]
\[ \beta_{33}(t)=0.022+0.001\sin
\bigg(\frac{1}{3}t\bigg),\,\alpha_{31}(t)=0.8+0.2|\sin\sqrt{3}t|,\]
\[
\alpha_{32}(t)=0.8+0.2\sin\sqrt{3}t, \alpha_{33}(t)=0.8+0.2\sin
\bigg(\frac{4}{3}t\bigg),
\]
\[
\tau_{31}(t)=e^{0.5|\cos( \pi t+\frac{3\pi}{2})|}, \tau_{32}(t)=e^{0.6|\cos(\pi t+\frac{3\pi}{2})|}, \tau_{33}(t)=e^{0.3|\sin 2\pi t|}.
\]
\end{example}

By calculating, we have
\[
c_1^-=0.2, c_1^+=0.22, b_{12}^-=0.02, b_{12}^+=0.04, b_{13}^-=0.05,
b_{13}^+=0.07,
\]
\[
\beta_{11}^-=0.05, \beta_{11}^+=0.09, \beta_{12}^-=0.14,
\beta_{12}^+=0.16, \beta_{13}^-=0.14, \beta_{13}^+=0.16,
\]
\[
\alpha_{11}^-=\alpha_{12}^-=\alpha_{13}^-=0.91,\,\,
\alpha_{11}^+=\alpha_{12}^+=\alpha_{13}^+=1,
\]
\[
c_2^-=0.28, c_2^+=0.32, b_{21}^-=0.04, b_{21}^+=0.06, b_{23}^-=0.04,
b_{23}^+=0.06,
\]
\[
\beta_{21}^-=0.05, \beta_{21}^+=0.07, \beta_{22}^-=0.03,
\beta_{22}^+=0.05, \beta_{23}^-=0.08, \beta_{23}^+=0.1,
\]
\[
\alpha_{21}^-=0.6, \alpha_{21}^+=1, \alpha_{22}^-=0.6,
\alpha_{22}^+=1, \alpha_{23}^-=0.6, \alpha_{23}^+=1,
\]
\[
c_3^-=0.4, c_3^+=0.43, b_{31}^-=0.15, b_{31}^+=0.17, b_{32}^-=12,
b_{32}^+=0.14,
\]
\[
\beta_{31}^-=0.01, \beta_{31}^+=0.03, \beta_{32}^-=0.022,
\beta_{32}^+=0.042, \beta_{33}^-=0.21, \beta_{33}^+=0.23,
\]
\[
\alpha_{31}^-=0.8, \alpha_{31}^+=1, \alpha_{32}^-=0.6,
\alpha_{32}^+=1, \alpha_{33}^-=0.6, \alpha_{23}^+=1.
\]
Hence,
\[\sum\limits_{k=1,k\neq
1}^3\frac{b_{1k}^+}{c_{1}^-}=\frac{b_{12}^+}{c_{1}^-}+\frac{b_{13}^+}{c_{1}^-}=\frac{0.04+0.07}{0.2}=0.55<1,\]
\[\sum\limits_{k=1,k\neq
2}^3\frac{b_{2k}^+}{c_{2}^-}=\frac{b_{21}^+}{c_{2}^-}+\frac{b_{23}^+}{c_{2}^-}=\frac{0.06+0.06}{0.28}=0.4286<1,\]
\[\sum\limits_{k=1,k\neq
3}^3\frac{b_{3k}^+}{c_{3}^-}=\frac{b_{31}^+}{c_{3}^-}+\frac{b_{32}^+}{c_{3}^-}=\frac{0.17+0.14}{0.4}= 0.775<1,\]
\[
b_{12}^++b_{13}^++\frac{\beta_{11}^+}{e^2}+\frac{\beta_{12}^+}{e^2}+\frac{\beta_{13}^+}{e^2}=0.04+0.07+\frac{0.09}{e^2}+\frac{0.16}{e^2}+\frac{0.16}{e^2}=0.1655<c_{1}^-=0.2,
\]
\[
b_{21}^++b_{23}^++\frac{\beta_{21}^+}{e^2}+\frac{\beta_{22}^+}{e^2}+\frac{\beta_{23}^+}{e^2}=0.06+0.06+\frac{0.07}{e^2}+\frac{0.05}{e^2}+\frac{0.07}{e^2}=0.1457<c_{2}^-=0.28,
\]
\[
b_{31}^++b_{32}^++\frac{\beta_{31}^+}{e^2}+\frac{\beta_{32}^+}{e^2}+\frac{\beta_{33}^+}{e^2}=0.03+0.04+\frac{0.06}{e^2}+\frac{0.33}{e^2}+\frac{0.23}{e^2}=0.1539<c_{3}^-=0.4,
\]

\[
A_2>\max\limits_{1\leq i\leq n}\bigg\{\bigg[1-\sum\limits_{k=1,k\neq
i}^n\frac{b_{ik}^+}{c_{i}^-}\bigg]^{-1}\sum\limits_{j=1}^n
\frac{\beta_{ij}^+}{c_i^-\alpha_{ij}^-e}\bigg\}=\max\limits_{1\leq i\leq n}\{2.7102,0.8431, 0.6449\}=2.7102.
\]
Let $A_2=2.72, \varsigma=0.7215, \min\{\alpha_{ij}^-\}=0.6,$ we have
\begin{eqnarray*}
\min\limits_{1\leq i\leq n}\bigg[1-\sum\limits_{k=1,k\neq
i}^n\frac{b_{ik}^-}{c_{i}^+}\bigg]^{-1}\sum\limits_{j=1}^n
A_2\frac{\beta_{ij}^-}{\alpha_{ij}^+c_{i}^+}e^{-\alpha_{ij}^+A_2}&=&\min\{1.2118,8.0234, 1.3407\}\\&=&1.2118>A_1>\frac{\varsigma}{\alpha_{ij}^-}\approx\frac{0.7215}{0.6}=1.2025.
\end{eqnarray*}

If $-c_i\in \mathcal{R}^+$, that is, $1-c_i(t)\mu(t)>0, i=1,2,3$, then it is easy to verify that all conditions of
Theorem \ref{th32} are satisfied. Therefore, the system in Example
4.1 has a unique positive almost periodic solution in the region $\mathbb{B}^*=\{\varphi|\,\,\varphi\in \mathbb{B}, A_1\leq
\varphi_i(t)\leq 2.72, t\in\mathbb{T},i=1,2,\ldots,n\}$, which is  exponentially stable.

Especially, if we take $\mathbb{T}=\mathbb{R}$ or $\mathbb{T}=\mathbb{Z}$, then  $1-c_i(t)\mu(t)>0, i=1,2,3$. Hence, in this case, the
continuous-time Nicholson's blowflies model \eqref{e1} and its discrete-time
analogue have the same dynamical behaviors (see Figures 1-8).

\begin{figure}
  \centering
  \includegraphics[width=10cm,height=6cm]{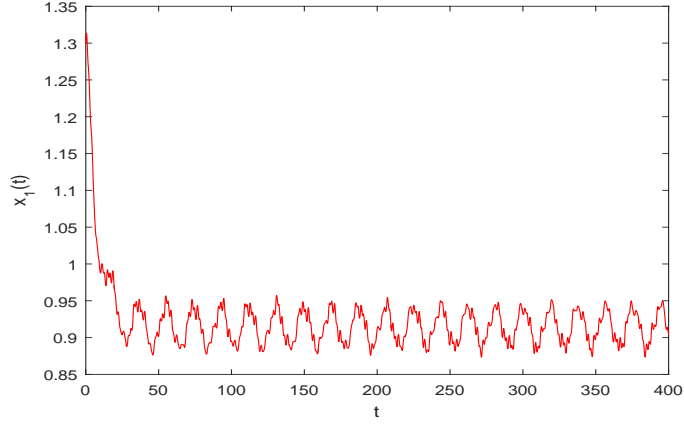}
  \caption{$\mathbb{T}=\mathbb{R}.$ Numerical solution $x_1(t)$ of system (4.1) for $(\varphi_1(t),\varphi_2(t),\varphi_3(t))=(1.3,1.3,1.5).$}
\end{figure}
\begin{figure}
\centering
  \includegraphics[width=10cm,height=6cm]{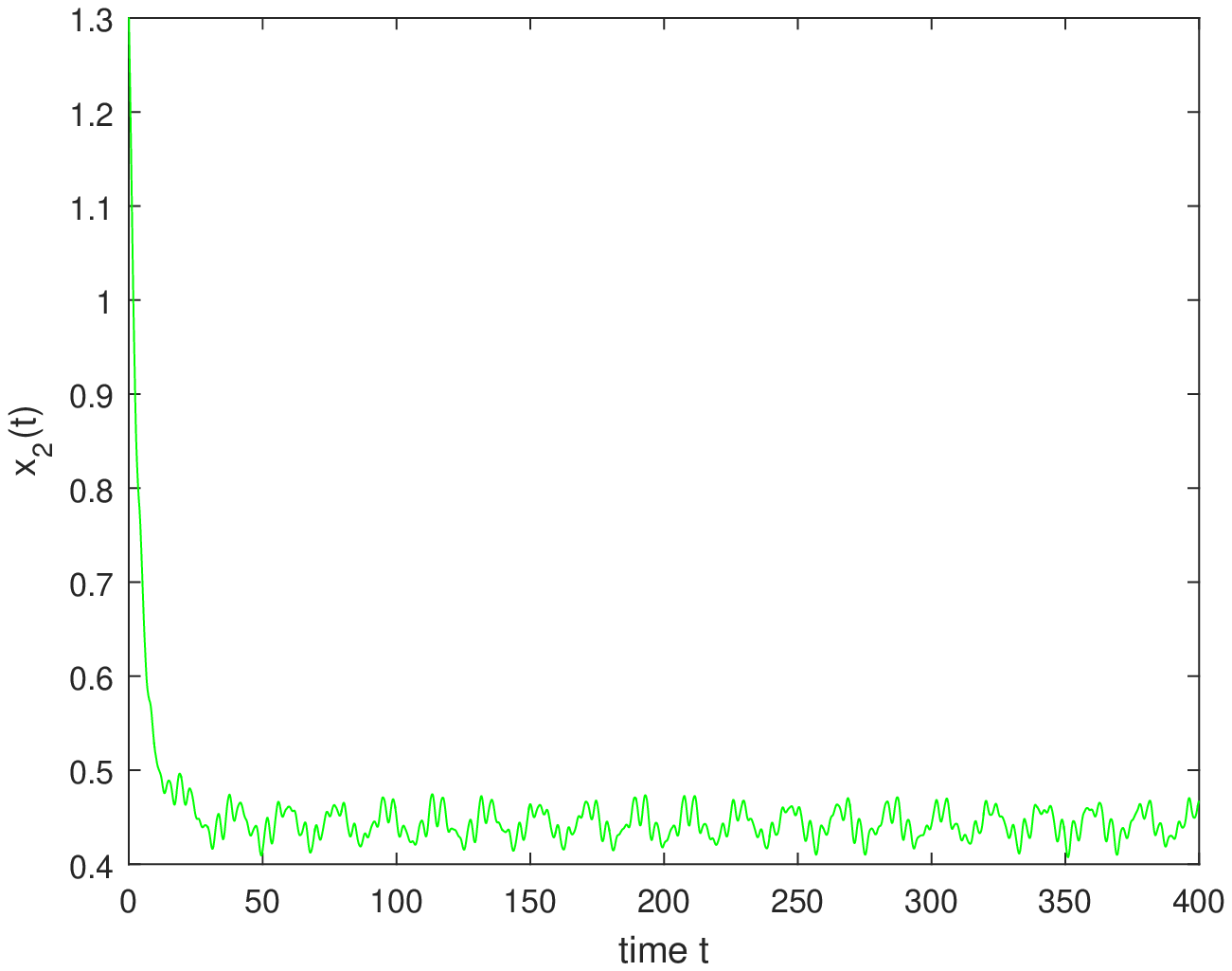}
  \caption{$\mathbb{T}=\mathbb{R}.$ Numerical solution $x_2(t)$ of system (4.1) for $(\varphi_1(t),\varphi_2(t),\varphi_3(t))=(1.3,1.3,1.5).$}
\end{figure}
\begin{figure}
\centering
  \includegraphics[width=10cm,height=6cm]{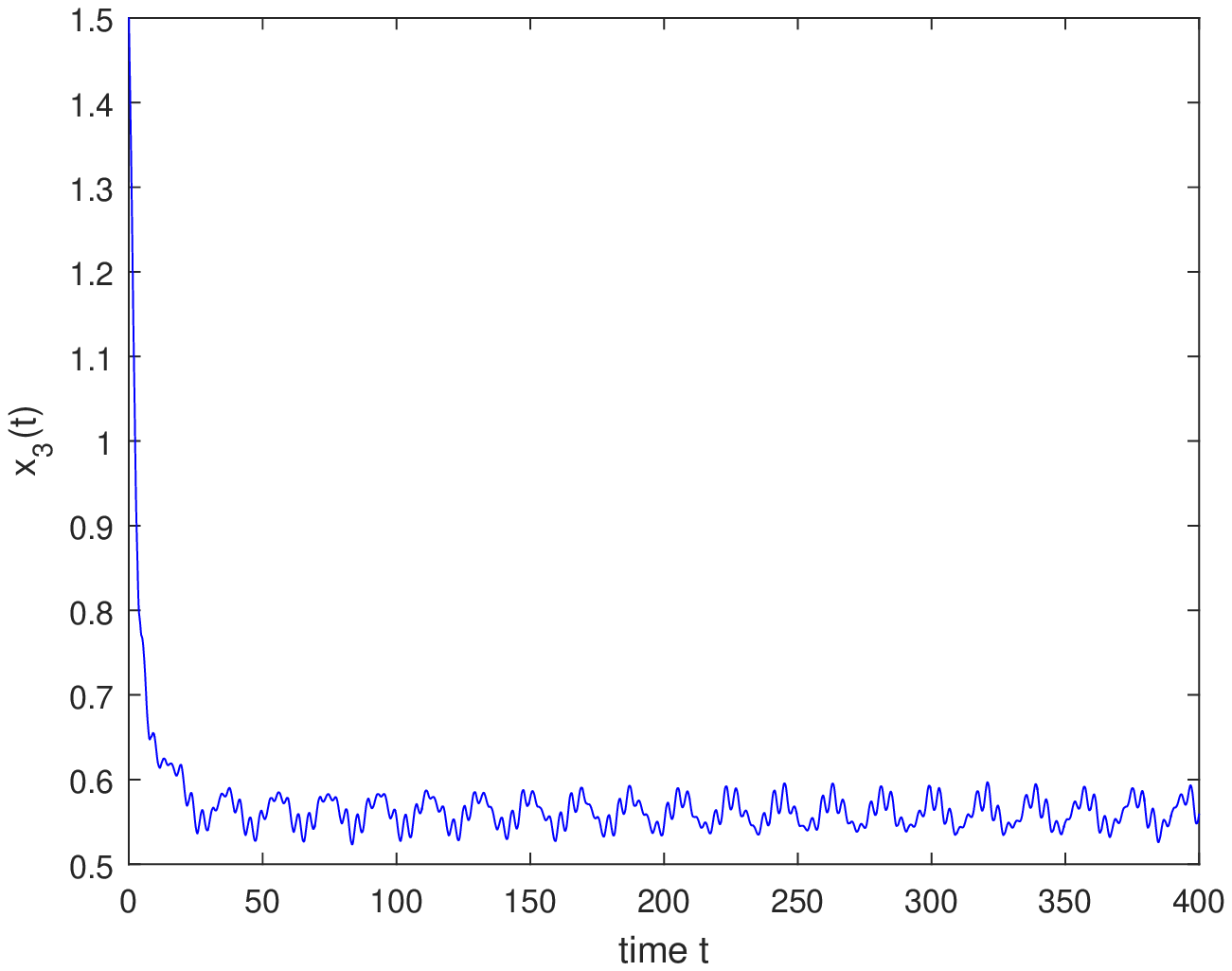}%
  \caption{$\mathbb{T}=\mathbb{R}.$ Numerical solution $x_3(t)$ of system (4.1) for $(\varphi_1(t),\varphi_2(t),\varphi_3(t))=(1.3,1.3,1.5).$}
\end{figure}

\begin{figure}
  \centering
  \includegraphics[width=10cm,height=6cm]{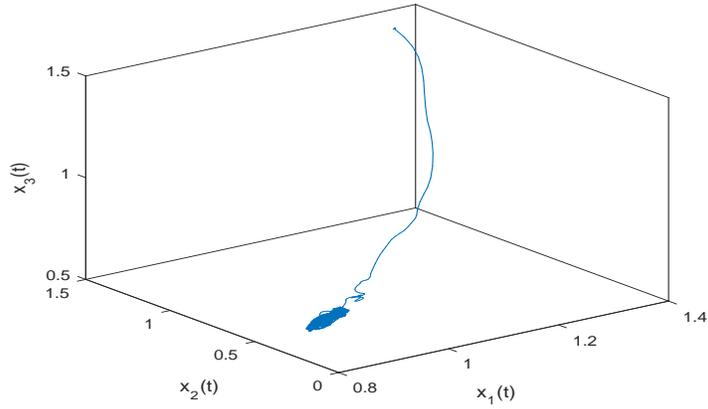}
  \caption{Continuous situation $(\mathbb{T}=\mathbb{R}): x_1(t), x_2(t),x_3(t)$.}
\end{figure}
\begin{figure}
  \centering
  \includegraphics[width=9cm,height=5cm]{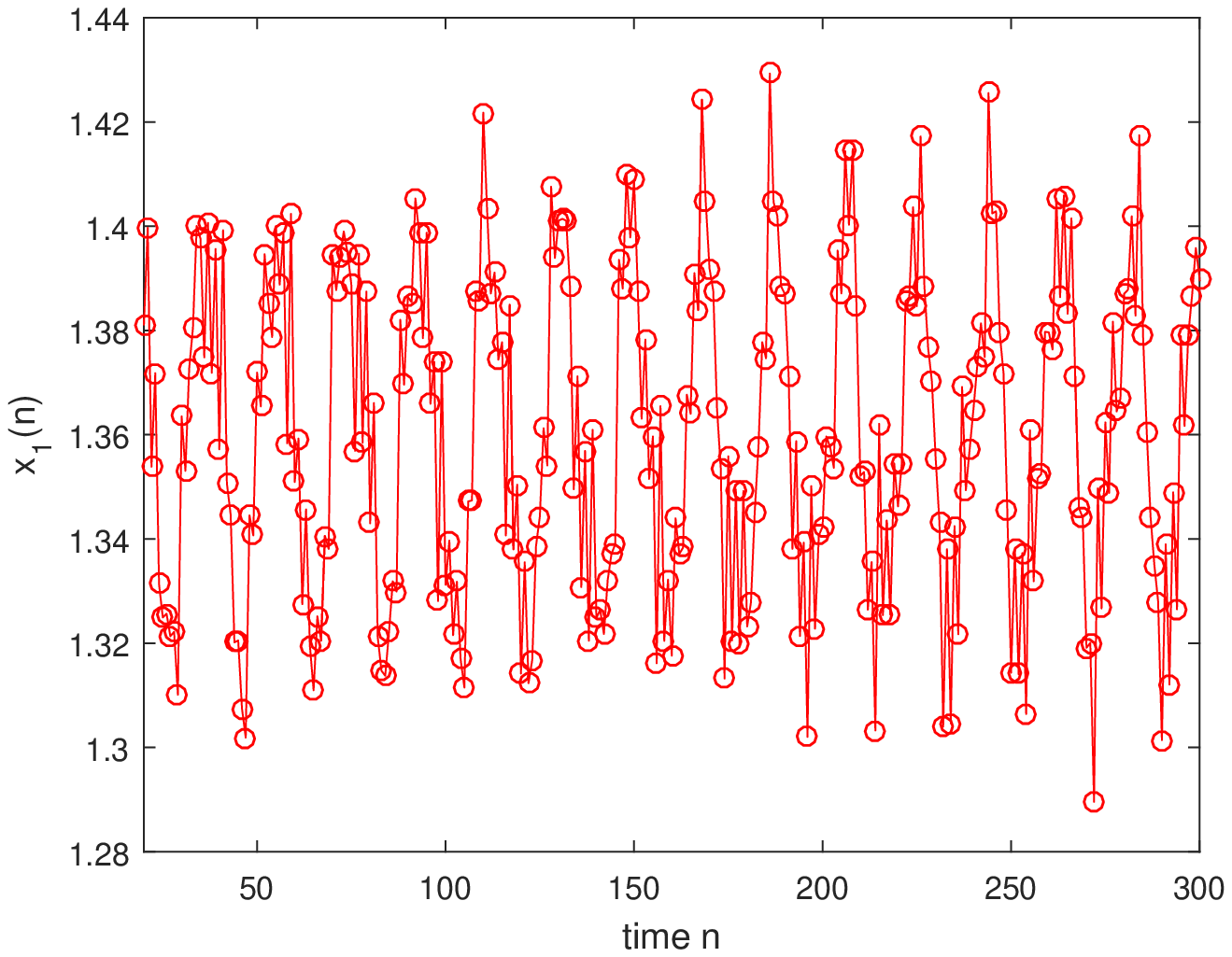}
  \caption{$\mathbb{T}=\mathbb{Z}.$ Numerical solution $x_1(n)$ of system (4.1) for $(\varphi_1(n),\varphi_2(n),\varphi_3(n))=(1.2,1.2,2.3).$}
\end{figure}
\begin{figure}
\centering
  \includegraphics[width=9cm,height=5cm]{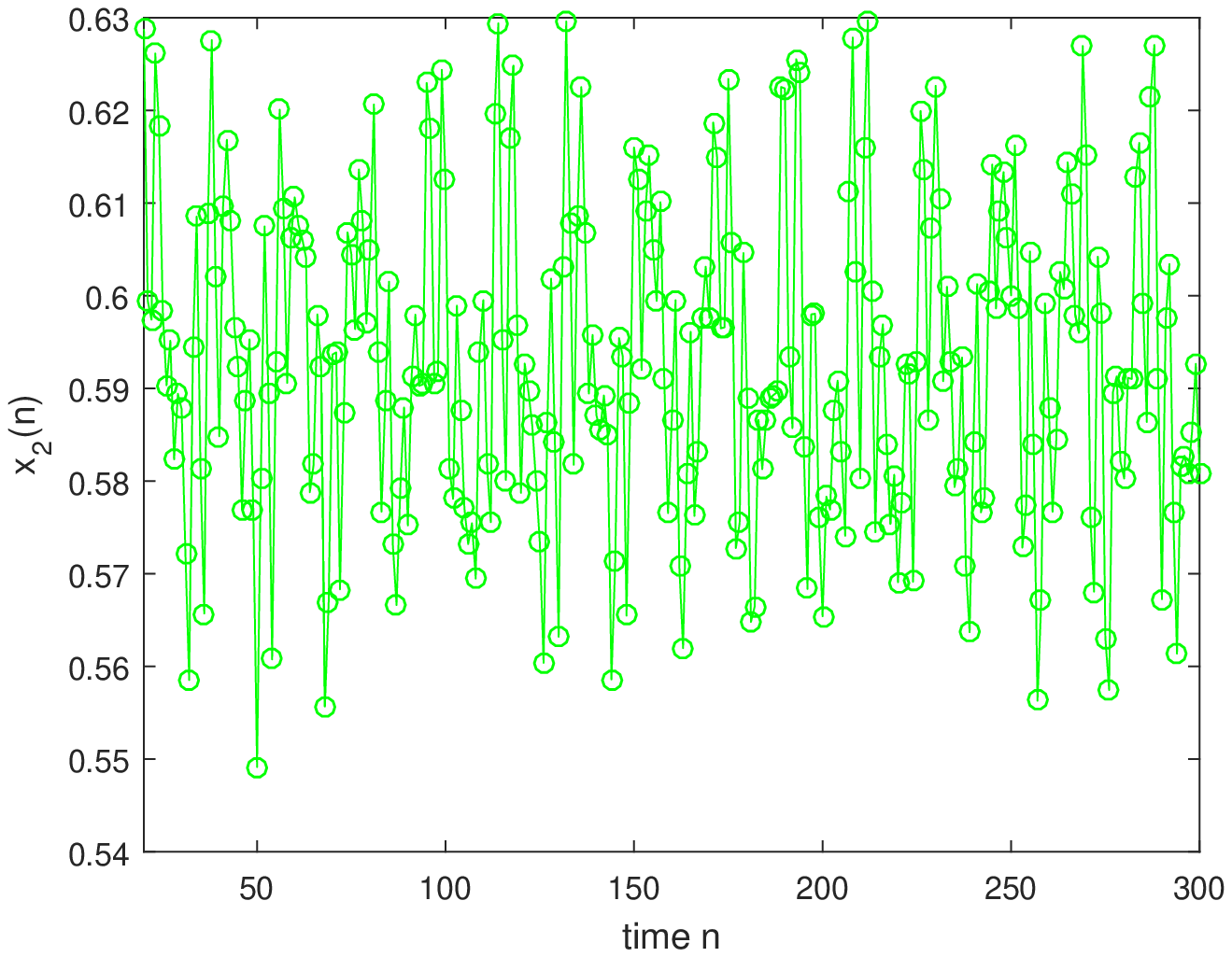}
  \caption{$\mathbb{T}=\mathbb{Z}.$ Numerical solution $x_2(n)$ of system (4.1) for $(\varphi_1(n),\varphi_2(n),\varphi_3(n))=(1.2,1.2,2.3).$}
\end{figure}
\begin{figure}
\centering
  \includegraphics[width=9cm,height=5cm]{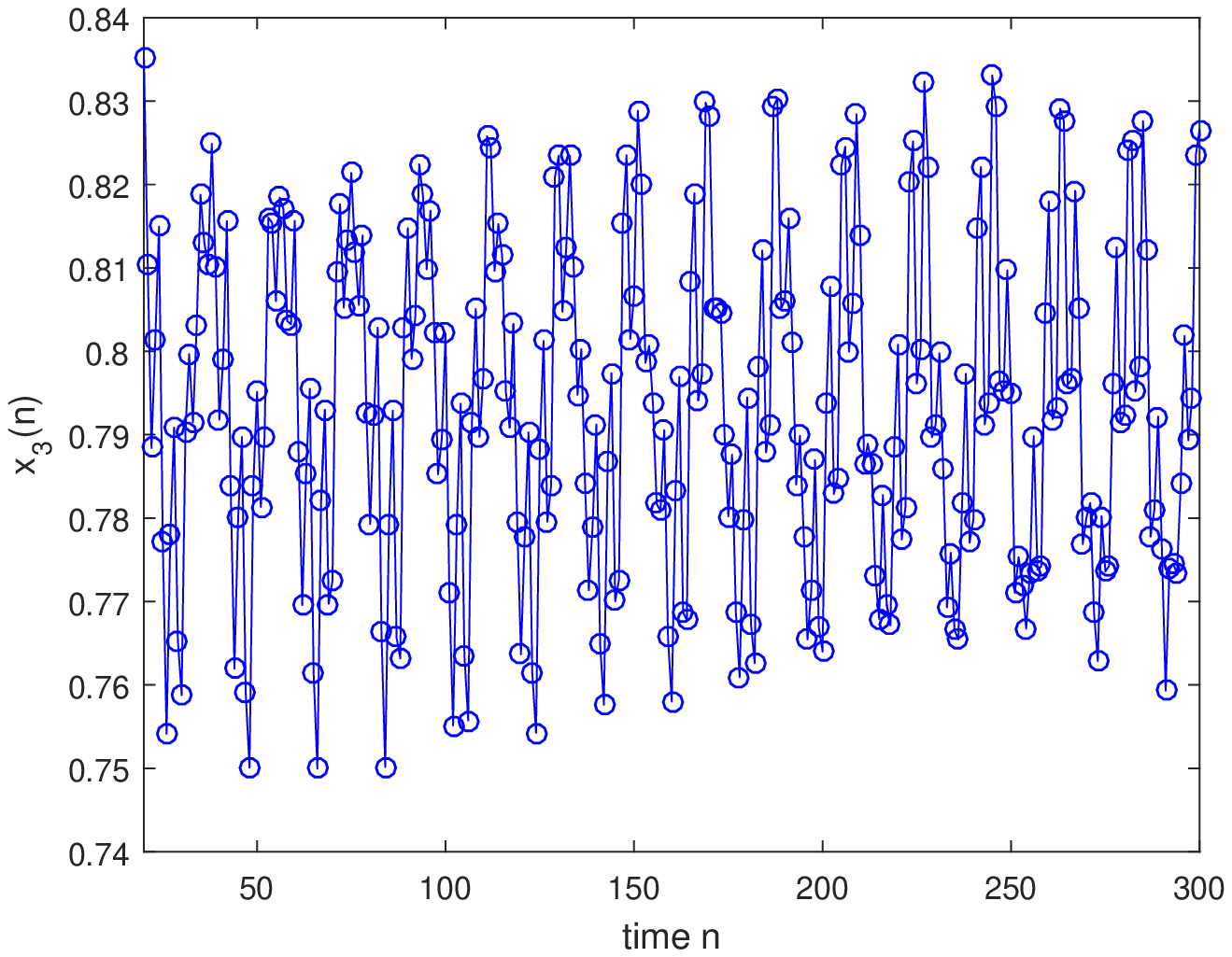}%
  \caption{$\mathbb{T}=\mathbb{Z}.$ Numerical solution $x_3(n)$ of system (4.1) for $(\varphi_1(t),\varphi_2(t),\varphi_3(n))=(1.2,1.2,2.3).$}
\end{figure}

\begin{figure}
  \centering
  \includegraphics[width=9cm,height=5cm]{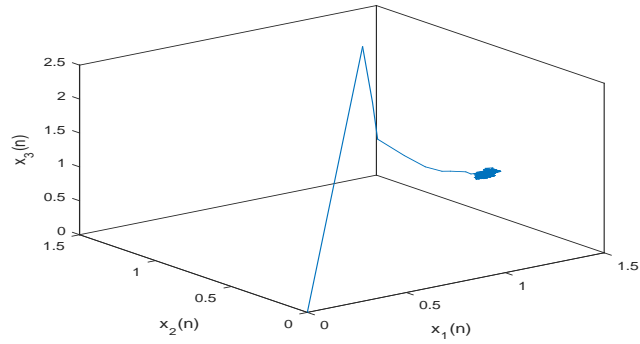}
  \caption{Discrete situation $(\mathbb{T}=\mathbb{R}): x_1(n), x_2(n),x_3(n)$.}
\end{figure}

\newpage
\section{Conclusion}
\indent

In this paper, we proposed a new  concept  of almost periodic time scales, two new definitions of almost periodic functions on time scales and  investigated
some  basic properties of them, which can unify the continuous and the  discrete cases effectively.
As an application,   we obtain some sufficient conditions for
the existence and  exponential stability of  positive almost periodic solutions
for a class of Nicholson's blowflies models on time scales.
 Our methods and results of this paper may be used to study almost periodicity of general dynamic equations on time scales.
Besides, based on our this new  concept  of almost periodic time scales, one  can further study  the problems of pseudo almost periodic functions,
 pseudo almost automorphic functions   and   pseudo almost periodic set-valued functions  on times as well as  the problems of pseudo almost periodic,
 pseudo almost automorphic    and  pseudo almost periodic set-valued dynamic systems  on times and so on.

\end{document}